\documentclass[a4paper,reqno]{amsart}

\usepackage{graphicx}
\usepackage{mathrsfs}
\usepackage{amsfonts}
\usepackage{amssymb}
\usepackage{amsmath}
\usepackage{amsthm}
\usepackage{amscd}
\usepackage[all,2cell]{xy}
\usepackage{color}
\usepackage[pagebackref,colorlinks]{hyperref}
\usepackage{enumerate}

\usepackage{comment}
\usepackage{hyperref}
\UseAllTwocells \SilentMatrices

\numberwithin{equation}{section}

\usepackage{comment}

\theoremstyle{plain}
\newtheorem{thm}{Theorem}[section]
\newtheorem{prop}[thm]{Proposition}
\newtheorem{cor}[thm]{Corollary}
\newtheorem{lem}[thm]{Lemma}

\theoremstyle{definition}
\newtheorem{defi}[thm]{Definition}
\newtheorem{exm}[thm]{Example}
\theoremstyle{remark}
\newtheorem{rmk}[thm]{\bf Remark}

\def\g{\gamma}
\def\G{\Gamma}
\def\mg{\mathcal{G}}
\def\xra{\xrightarrow[]{}}
\def\a{\alpha}
\def\b{\beta}

\def\AA{\mathcal{A}}
\def\VV{\mathcal{V}}
\def \Z{\mathbb Z}
\def\-{\text{-}}
\def\TT{\mathcal{T}}
\def\c{\circ}

\newcommand{\End}{\operatorname{End}}
\newcommand{\END}{\operatorname{END}}
\newcommand{\Hom}{\operatorname{Hom}}

\newcommand{\Ann}{\operatorname{Ann}}
\newcommand{\supp}{\operatorname{supp}}
\newcommand{\Iso}{\operatorname{Iso}}

\newcommand{\Ker}{\operatorname{Ker}}

\newcommand{\gr}{\operatorname{gr}}
\newcommand{\KP}{\operatorname{KP}}
\newcommand{\Mod}{\operatorname{Mod}}
\newcommand{\Gr}{\operatorname{Gr}}

\begin{document}

\title{Graded Steinberg algebras and their representations}

\author{Pere Ara}
\address{Department of Mathematics\\
Universitat Auto\'noma de Barcelona\\
 08193 Bellaterra (Barcelona), Spain}
\email{para@mat.uab.cat}

\author{Roozbeh Hazrat}
\address{
Centre for Research in Mathematics\\
Western Sydney University\\
Australia} \email{r.hazrat@westernsydney.edu.au}

\author{Huanhuan Li}
\address{
Centre for Research in Mathematics\\
Western Sydney University\\
Australia} \email{h.li@westernsydney.edu.au}

\author{Aidan Sims}
\address{School of Mathematics and Applied Statistics\\
University of Wollongong\\
NSW 2522, Australia} \email{asims@uow.edu.au}



\date{\today}

\begin{abstract}
We study the category of left unital graded modules over the Steinberg algebra of a
graded ample Hausdorff groupoid. In the first part of the paper, we show that this
category is isomorphic to the category of unital left modules over the Steinberg algebra
of the skew-product groupoid arising from the grading. To do this, we show that the
Steinberg algebra of the skew product is graded isomorphic to a natural generalisation of
the the Cohen-Montgomery smash product of the Steinberg algebra of the underlying
groupoid with the grading group. In the second part of the paper, we study the minimal
(that is, irreducible) representations in the category of graded modules of a Steinberg
algebra, and establish a connection between the annihilator ideals of these minimal
representations, and effectiveness of the groupoid.

Specialising our results, we produce a representation of the monoid of graded finitely
generated projective modules over a Leavitt path algebra. We deduce that the lattice of
order-ideals in the $K_0$-group of the Leavitt path algebra is isomorphic to the lattice
of graded ideals of the algebra. We also investigate the graded monoid for Kumjian--Pask
algebras of row-finite $k$-graphs with no sources. We prove that these algebras are
graded von Neumann regular rings, and record some structural consequences of this.
\end{abstract}

\maketitle


\section{Introduction}
There has long been a trend of ``algebraisation'' of concepts from operator theory into
algebra. This trend seems to have started with von Neumann and Kaplansky and their
students Berberian and Rickart to see what properties in operator algebra theory arise
naturally from discrete underlying structures~\cite{kap}. As Berberian puts
it~\cite{berber}, ``if all the functional analysis is stripped away\dots what remains
should stand firmly as a substantial piece of algebra, completely accessible through
algebraic avenues''.

In the last decade, Leavitt path algebras~\cite{ap,amp} were introduced as an
algebraisation of graph $C^*$-algebras~\cite{kprr, raeburn} and in particular
Cuntz--Krieger algebras. Later, Kumjian--Pask algebras~\cite{pchr} arose as an
algebraisation of higher-rank graph $C^*$-algebras~\cite{kp2000}. Quite recently
Steinberg algebras were introduced in~\cite{st,cfst} as an algebraisation of the groupoid
$C^*$-algebras first studied by Renault~\cite{re}. Groupoid $C^*$-algebras include all
graph $C^*$-algebras and higher-rank graph $C^*$-algebras, and Steinberg algebras include
Leavitt and Kumjian--Pask algebras as well as inverse semigroup algebras. More generally,
groupoid $C^*$-algebras provide a model for inverse-semigroup $C^*$-algebras, and the
corresponding inverse-semigroup algebras are the Steinberg algebras of the corresponding
groupoids. All of these classes of algebras have been attracting significant attention,
with particular interest in whether $K$-theoretic data can be used to classify various
classes of Leavitt path algebras, inspired by the Kirchberg--Phillips classification
theorem for $C^*$-algebras~\cite{phillips}.

In this note we study graded representations of Steinberg algebras. For a $\G$-graded
groupoid $\mg$, (i.e., a groupoid $\mg$ with a cocycle map $c:\mg \rightarrow \G$)
Renault proved~\cite[Theorem~5.7]{re} that if $\G$ is a discrete abelian group with
Pontryagin dual $\widehat{\G}$, then the $C^*$-algebra $C^*(\mg\times_c \G)$ of the
skew-product groupoid is isomorphic to a crossed-product $C^*$-algebra $C^*(\mg)\times
\widehat{\G}$. Kumjian and Pask~\cite{kp} used Renault's results to show that if there is
a free action of a group $\G$ on a graph $E$, then the crossed product of graph
$C^*$-algebra  by the induced action is strongly Morita equivalent to $C^*(E/\G)$, where
$E/\G$ is the quotient graph.

Parallelling Renault's work, we first consider the Steinberg algebras of skew-product
groupoids (for arbitrary discrete groups $\G$). We extend Cohen and Montgomery's
definition of the smash product of a graded ring by the grading group (introduced and
studied in their seminal paper~\cite{cm1984}) to the setting of non-unital rings. We then
prove that the Steinberg algebra of the skew-product groupoid is isomorphic to the
corresponding smash product. This allows us to relate the category of graded modules of
the algebra to the category of modules of its smash product. Specialising to Leavitt path
algebras, the smash product by the integers arising from the canonical grading yields an
ultramatricial algebra. This allows us to give a presentation of the monoid of graded
finitely generated projective modules for Leavitt path algebras of arbitrary graphs. In
particular, we prove that this monoid is cancellative. The group completion of this
monoid is called the graded Grothendieck group, $K^{\gr}_0$, which is a crucial invariant
in study of Leavitt path algebras.  It is conjectured~\cite[\S3.9]{haz} that the graded
Grothendieck group is a complete invariant for Leavitt path algebras. We study the
lattice of order ideals of $K^{\gr}_0$ and establish a lattice isomorphism between order
ideals of $K^{\gr}_0$ and graded ideals of Leavitt path algebras.

We then apply the smash product to Kumjian--Pask algebras $\KP_K(\Lambda)$. Unlike
Leavitt path algebras, Kumjian--Pask algebras of arbitrary higher rank graphs are poorly
understood, so we restrict our attention to row finite $k$-graphs with no sources. We
show that the smash product of $\KP_K(\Lambda)$ by $\Z^k$ is also an ultramatricial
algebra. This allows us to show that $\KP_K(\Lambda)$ is a graded von Neumann regular
ring and, as in the case of Leavitt path algebras, its graded monoid is cancellative.
Several very interesting properties of Kumjian--Pask algebras follow as a consequence of
general results for graded von Neumann regular rings.

We then proceed with a systematic study of the irreducible representations of Steinberg
algebras. In~\cite{c}, Chen used infinite paths in a graph $E$ to construct an
irreducible representation of the Leavitt path algebra $E$. These representations were
further explored in a series of papers~\cite{abramsman, araranga, araranga2, hr, ranga}.
The infinite path representations of Kumjian--Pask algebras were also defined in
\cite{pchr}. In the setting of a groupoid $\mg$, the infinite path space becomes the unit
space of the groupoid. For any invariant subset $W$ of the unit space, the free module
$RW$ with basis $W$ is a representation of the Steinberg algebra
$A_{R}(\mg)$~\cite{bcfs}. These representations were used to construct nontrivial ideals
of the Steinberg algebra, and ultimately to characterise simplicity.

For the $\G$-graded groupoid $\mg$,  we introduce what we call $\G$-aperiodic invariant
subsets of the unit space of the groupoid $\mg$. We obtain graded (irreducible)
representations of the Steinberg algebra via these $\G$-aperiodic invariant subsets. We
then describe the annihilator ideals of these graded representations and establish a
connection between these annihilator ideals and effectiveness of the groupoid.
Specialising to the case of Leavitt and Kumjian--Pask algebras we obtain new results
about representations of these algebras.

The paper is organised as follows. In Section~\ref{grsmash}, we recall the background we
need on graded ring theory, and then introduce the smash product $A\#\G$ of an arbitrary
$\G$-graded ring $A$, possibly without unit. We establish an isomorphism of categories
between the category of unital left $A\#\G$-modules and the category of unital left
$\G$-graded $A$-modules. This theory is used in Section~\ref{sectionthree}, where we
consider the Steinberg algebra associated to a $\G$-graded ample groupoid $\mg$. We prove
that the Steinberg algebra of the skew-product of $\mg\times_c \G$ is graded isomorphic
to the smash product of $A_{R}(\mg)$ with the group $\G$.

In Section \ref{sectiondrdr} we collect the facts we need to study the monoid of graded
rings with graded local units. In Section \ref{sectionfive} and Section \ref{sectionsix},
we apply the isomorphism of categories in Section \ref{grsmash} and the graded
isomorphism of Steinberg algebras (Theorem \ref{isothm}) on the setting of Leavitt path
algebras and Kumjian--Pask algebras. Although Kumjian--Pask algebras are a generalisation
of Leavitt path algebras, we treat these classes separately as we are able to study
Leavitt path algebras associated to any arbitrary graph, whereas for Kumjian--Pask
algebras we consider only row-finite $k$-graphs with no sources, as the general case is
much more complicated \cite{rsy, si2}. We describe the monoids of graded finitely
generated projective modules over Leavitt path algebras and Kumjian--Pask algebras, and
obtain a new description of their lattices of graded ideals. In Section
\ref{sectionseven}, we turn our attention to the irreducible representations of Steinberg
algebras. We consider what we call $\G$-aperiodic invariant subset of the groupoid $\mg$
and construct graded simple $A_{R}(\mg)$-modules. This covers, as a special case,
previous work done in the setting of Leavitt path algebras, and gives new results in the
setting of Kumjian--Pask algebras. We describe the annihilator ideals of the graded
modules over a Steinberg algebra and prove that these ideals reflect the effectiveness of
the groupoid.

\section{Graded rings and smash products}\label{grsmash}

\subsection{Graded rings}\label{grtheory}

Let $\Gamma$ be a group with identity $\varepsilon$. A ring $A$ (possibly without unit)
is called a \emph{$\Gamma$-graded ring} if $ A=\bigoplus_{ \gamma \in \Gamma} A_{\gamma}$
such that each $A_{\gamma}$ is an additive subgroup of $A$ and $A_{\gamma}  A_{\delta}
\subseteq A_{\gamma\delta}$ for all $\gamma, \delta \in \Gamma$. The group $A_\gamma$ is
called the $\gamma$-\emph{homogeneous component} of $A.$ When it is clear from context
that a ring $A$ is graded by group $\Gamma,$ we simply say that $A$ is a  \emph{graded
ring}. If $A$ is an algebra over a ring $R$, then $A$ is called a \emph{graded algebra}
if $A$ is a graded ring and $A_{\gamma}$ is a $R$-submodule for any $\gamma \in \Gamma$.
A $\G$-graded ring $A=\bigoplus_{\g\in\G}A_{\g}$is called \emph{strongly graded} if
$A_{\g}A_{\delta}=A_{\g\delta}$ for all $\g,\delta$ in $\G$.

The elements of $\bigcup_{\gamma \in \Gamma} A_{\gamma}$ in a graded ring $A$ are called
\emph{homogeneous elements} of $A.$ The nonzero elements of $A_\gamma$ are called
\emph{homogeneous of degree $\gamma$} and we write $\deg(a) = \gamma$ for $a \in
A_{\gamma}\backslash \{0\}.$ The set $\Gamma_A=\{ \gamma \in \Gamma \mid A_\gamma \not =
0 \}$ is called the \emph{support}  of $A$. We say that a $\Gamma$-graded ring $A$ is
\emph{trivially graded} if the support of $A$ is the trivial group
$\{\varepsilon\}$---that is, $A_\varepsilon=A$, so $A_\gamma=0$ for $\gamma \in \Gamma
\backslash \{\varepsilon\}$. Any ring admits a trivial grading by any group. If $A$ is a
$\G$-graded ring and $s \in A$, then we write $s_\alpha, \alpha \in \G$ for the unique
elements $s_\alpha \in A_\alpha$ such that $s = \sum_{\alpha \in \G} s_\alpha$. Note that
$\{\alpha \in \G : s_\alpha \not= 0\}$ is finite for every $s \in A$.

We say a $\G$-graded ring $A$ has \emph{graded local units} if for any finite set of
homogeneous elements  $\{x_{1}, \cdots, x_{n}\}\subseteq A$, there exists a homogeneous
idempotent $e\in A$ such that $\{x_{1}, \cdots, x_{n}\}\subseteq eAe$. Equivalently, $A$
has graded local units, if $A_\varepsilon$ has local units and $A_\varepsilon
A_{\g}=A_{\g}A_\varepsilon=A_{\g}$ for every $\g \in \G$.

Let $M$ be a left $A$-module. We say $M$ is unital if $AM=M$ and it is \emph{$\G$-graded}
if there is a decomposition $M=\bigoplus_{\g\in\G}M_{\g}$ such that
$A_{\a}M_{\g}\subseteq M_{\a\g}$ for all $\a,\g \in \G$. We denote by $A$-$\mathrm{Mod}$
the category of unital left $A$-modules and by $A$-$\mathrm{Gr}$ the category of
$\G$-graded unital left $A$-modules with morphisms the $A$-module homomorphisms that
preserve grading.

For a graded left $A$-module $M$, we define the $\a$-\emph{shifted} graded left
$A$-module $M(\a)$ as
\begin{equation}\label{eq:M-shifted}
M(\a)=\bigoplus_{\g\in \G}M(\a)_{\g},
\end{equation}
where $M(\a)_{\g}=M_{\g\a}$. That is, as an ungraded module, $M(\alpha)$ is a copy of
$M$, but the grading is shifted by $\alpha$. For $\a\in\G$, the \emph{shift functor}
\begin{equation*}
\mathcal{T}_{\a}: A\text{-}{\Gr}\longrightarrow A\text{-}{\Gr},\quad M\mapsto M(\a)
\end{equation*}
is an isomorphism with the property $\mathcal{T}_{\a}\mathcal{T}_{\b}=\mathcal{T}_{\a\b}$
for $\a,\b\in\G$.

 \subsection{Smash products}\label{smashp}

Let $A$ be a $\Gamma$-graded unital $R$-algebra where $\Gamma$ is a finite group. In the
influential paper~\cite{cm1984}, Cohen and Montgomery introduced the smash product
associated to $A$, denoted by $A\#R[\Gamma]^*$. They proved two main theorems, duality
for actions and coactions, which related the smash product to the ring $A$. In turn,
these theorems relate the graded structure of $A$ to non-graded properties of $A$. The
construction has been extended to the case of infinite groups (see for
example~\cite{be,shvan} and \cite[\S7]{grrings}). We need to adopt the construction of
smash products for algebras with local units as the main algebras we will be concerned
with are Steinberg algebras which are not necessarily unital but have local units. The
main theorem of Section~\ref{sectionthree}  shows that the Steinberg algebra of the
skew-product of a groupoid by a group can be represented using the smash product
construction (Theorem~\ref{isothm}).

We start with a general definition of smash product for any ring.

\begin{defi}
For a $\G$-graded ring $A$ (possibly without unit), the \emph{smash product} ring $A\#\G$
is defined as the set of all formal sums $\sum_{\gamma \in \G} r^{(\gamma)} p_\gamma $,
where $r^{(\gamma)}\in A$ and $p_\gamma$ are symbols. Addition is defined component-wise
and multiplication is defined by linear extension of the rule
$(rp_{\a})(sp_{\b})=rs_{\a\b^{-1}}p_{\b},$  where $r,s\in A$ and $\a,\b\in\G$.
\end{defi}

It is routine to check that $A\#\G$ is a ring. We emphasise that the symbols $p_{\g}$ do
not belong to $A\#\G$; however if the ring $A$ has unit, then we regard the $p_\g$ as
elements of $A\#\G$ by identifying $1_A p_\g$ with $p_\g$. Each $p_{\g}$ is then an
idempotent element of $A\#\G$. In this case $A\#\G$ coincides with the ring $A\#\G^{*}$
of \cite{be}. If $\G$ is finite, then $A\#\G$ is the same as the smash product
$A\#k[\G]^{*}$ of \cite{cm1984}.  Note that $A\#\G$ is always a $\Gamma$-graded ring with
\begin{equation}\label{sydhyhy}
(A\#\G)_\gamma=\sum_{\alpha \in \G}A_\gamma p_\alpha.
\end{equation}

Next we define a shift functor on $A\#\G\-\Mod$. This functor will coincide with the
shift functor on $A$-$\mathrm{Gr}$ (see Proposition~\ref{kikiki}).  This does not seem to
be exploited in the literature  and will be crucial in our study of $K$-theory of Leavitt
path algebras (\S\ref{stlib}).

For each $\a\in \G$, there is an algebra automorphism
\[
\mathcal{S}^{\a}: A\#\G \longrightarrow A\#\G
\]
such that $\mathcal{S}^{\a}(sp_{\b})=sp_{\b\a}$ for $sp_{\b}\in A\#\G$ with $s\in A$ and
$\b\in\G$. We sometimes call $\mathcal{S}^\a$ the \emph{shift map} associated to $\a$.
For $M\in A\#\G\-\Mod$ and $\a\in\G$, we obtain a shifted $A\#\G$-module
$\mathcal{S}^\alpha_* M$ obtained by setting $\mathcal{S}^\alpha_* M := M$ as a group,
and defining the left action by $a \cdot_{\mathcal{S}^\alpha_* M} m :=
\mathcal{S}^\alpha(a) \cdot_M m$. For $\a\in\G$, the \emph{shift functor}
\begin{equation*}
\widetilde{\mathcal{S}}^{\a}: A\#\G\text{-}\Mod \longrightarrow A\#\G\-\Mod,\quad M \mapsto  \mathcal{S}^\alpha_* M
\end{equation*}
is an isomorphism satisfying $\widetilde{\mathcal{S}}_\a \widetilde{\mathcal{S}}_\b =
\widetilde{\mathcal{S}}_{\a\b}$ for $\a,\b\in\G$.

If $A$ is a unital ring then $A\#\G$ has local units (\cite[Proposition 2.3]{be})). We
extend this to rings with graded local units.

\begin{lem}\label{lemma-hl}
Let $A$ be a $\Gamma$-graded ring with graded local units. Then the ring $A\#\G$ has
graded local units.
\end{lem}
\begin{proof}
Take a finite subset $X=\{x_{1}, x_{2}, \cdots x_{n}\}\subseteq A\#\G$ such that all
$x_{i}$ are homogeneous elements. Since homogeneous elements of $A\#\G$ are sums of
elements of the form $rp_{\a}$ for $r\in A$ a homogeneous element and $\a\in\G$, we may
assume that $x_{i}=r_{i}p_{\a_{i}}$, $1\le i \leq n$, where $r_{i}\in A$ are homogeneous
of degree $\g_{i}$ and $\a_{i}\in \G$. Since $A$ has graded local units, there exists a
homogenous idempotent $e\in A$ such that $er_{i}=r_{i}e=r_{i}$ for all $i$. Consider the
finite set \begin{equation*} Y=\{\g\in \G \mid  \g=\a_{i} \text{~or~} \g=\g_{i}\a_{i}
\text{~for~}  1\leq i \leq n\},
\end{equation*} and let $w=\sum_{\g\in Y}ep_{\g}$. Since the idempotent $e\in A$ is homogeneous, $w$ is a
homogeneous element  of $A\#\G$. It is easy now to check that  $w^{2}=w$ and $wx_{i}= x_i
=x_{i}w$ for all $i$.
\end{proof}

As we will see in Sections~\ref{sectionfive} and \ref{sectionsix}, smash products of
Leavitt path algebras or of Kumjian--Pask algebras are ultramatricial algebras, which are
very well-behaved. This allows us to obtain results about the path algebras via their
smash product. For example, ultramatricial algebras are von Neumann regular rings. The
following lemma allows us to exploit this property (see Theorems~\ref{poryt},
\ref{poryt1}). Recall that a graded ring is called \emph{graded von Neumann regular} if
for any homogeneous element $a$, there is an element $b$ such that $aba=a$.

\begin{lem}\label{grneumann}
Let $A$ be a $\G$-graded ring (possibly without unit). Then $A\# \G$ is graded von
Neumann regular if and only if $A$ is graded von Neumann regular.
\end{lem}
\begin{proof}
Suppose $A\#\G$ is graded regular and $a\in A_\gamma$, for some $\gamma \in G$. Since
$ap_e \in (A\#\G)_\g$ (see~(\ref{sydhyhy})), there is an element $\sum_{\alpha\in \G}
b_{\gamma_\alpha} p_\alpha \in (A\#\G)_{\gamma^{-1}}$ with
$\deg(b_{\gamma_\alpha})=\gamma^{-1}$, $\alpha \in \G$, such that
\[
    ap_e\Big(\sum_{\alpha\in \G} b_{\gamma_\alpha} p_\alpha\Big) ap_e=ap_e.
\]
This identity reduces to $ab_{\gamma_\g} ap_e=ap_e$. Thus $ab_{\gamma_\g} a=a$. This
shows that $A$ is graded regular.

Conversely, suppose $A$ is graded regular and $x:=\sum_{\alpha\in \G} a_{\gamma_\alpha}
p_\alpha \in (A\#\G)_\gamma$. By (\ref{sydhyhy}) we have
$\deg(a_{\gamma_\alpha})=\gamma$, $\alpha \in \G$.  Then there are
$b_{\gamma^{-1}_\alpha} \in A_{\gamma^{-1}}$ such that
$a_{\gamma_\alpha}b_{\gamma^{-1}_\alpha}a_{\gamma_\alpha}=a_{\gamma_\alpha}$, for $\alpha
\in \G$. Consider the element $y:=\sum_{\alpha\in \G} b_{\gamma^{-1}_\alpha} p_{\g\alpha}
\in (A\#\G)_{\gamma^{-1}}$. One can then check that $xyx=x$. Thus $A\#\G$ is graded
regular.
\end{proof}

\subsection{An isomorphism of module categories} \label{subsection31}
In this section we first prove that, for a $\Gamma$-graded ring $A$ with graded local
units,  there is an isomorphism between the categories $A\#\G$-$\mathrm{Mod}$ and
$A$-$\mathrm{Gr}$ (Proposition~\ref{proposition}). This is a generalisation
of~\cite[Theorem~2.2]{cm} and~\cite[Theorem~2.6]{be}. We check that the isomorphism
respects the shifting in these categories. This in turn translates the shifting of
modules in the category of graded modules to an action of the group on the category of
modules for the smash-product. Since graded Steinberg algebras have graded local units,
using this result and Theorem~\ref{isothm}, we obtain a shift preserving isomorphism
\[
    A_R(\mg\times_c \G) \text{-} \mathrm{Mod} \cong A_R(\mg) \text{-} \mathrm{Gr}.
\]

In Section \ref{sectionfive} we will use this in the setting of Leavitt path algebras to
establish an isomorphism between the category of graded modules of $L_{R}(E)$ and the
category of modules of $L_{R}(\overline E)$, where $\overline E$ is the covering graph of
$E$ (\S\ref{covergraph}). This yields a presentation of the monoid of  graded finitely
generated projective modules of a Leavitt path algebra.

We start with the following fact, which extends~\cite[Corollary 2.4]{be} to rings with
local units.

\begin{lem} \label{lemma-unital}
Let $A$ be a $\G$-graded ring with a set of graded local units $E$. A left $A\#\G$-module
$M$ is unital if and only if for every finite subset $F$ of $M$, there exists
$w=\sum_{i=1}^{n}up_{\g_{i}}$ with $\g_{i}\in \G$, and $u\in E$ such that $wx=x$ for all
$x\in F$.
\end{lem}

\begin{proof} Suppose that $M$ is unital. Then each $m\in F$ may be written as
$m = \sum_{n \in G_m} y_n n$ for some finite $G_m \subseteq M$ and choice of scalars
$\{y_n : n \in G_m\} \subseteq A\#\G$. Let $T := \bigcup_{m \in F} G_m$. By
Lemma~\ref{lemma-hl}, there exists a finite set $Y$ of $\G$ such that $w=\sum_{\g\in
Y}up_{\g}$ satisfies $wy=y$ for all $y\in T$. So $wm=m$ for all $m\in F$.

Conversely, for $m\in M$, take $F=\{m\}$. Then there exists $w$ such that $m=wm\in
(A\#\G)M$; that is, $(A\#\G)M = M$.
\end{proof}

\begin{prop} \label{proposition}\label{kikiki}
Let $A$ be a $\G$-graded ring with graded local units. Then there is an isomorphism of
categories $A\-\mathrm{Gr}\xrightarrow[]{\sim} A\#\G$-$\mathrm{Mod} $ such that the following
diagram commutes for every $\a \in \G$.
\begin{equation}\label{inducedshift}
\xymatrix{ A\-\Gr\ar[r]^(.4){\sim}\ar[d]^{\mathcal{T}_{\a}}&A\#\G\-\Mod\ar[d]^{\widetilde{\mathcal{S}}^{\a}}\\
A\-\Gr\ar[r]^(.4){\sim}&A\#\G\-\Mod
}
\end{equation}
\end{prop}
\begin{proof} We first define a functor $\phi: A\#\G\-\Mod\xra A\-\Gr$ as follows. Fix a set $E$ of graded local units for $A$. Let $M$ be a unital left $A\#\G$-module. We
view $M$ as a $\G$-graded left $A$-module $M'$ as follows. For each $\g \in \G$, define
\[
M_{\g}' := \sum_{u\in E}up_{\g}M.
\]

We first show that for $\a \in \G$, we have $M_{\a}' \cap \sum_{\g\in\G, \g\not =
\a}M_{\g}'= \{0\}$. Suppose this is not the case, so there exist finite index sets $F$
and $\{F'_\g : \g \in \G\}$ (only finitely many nonempty), elements $\{u_i : i \in F\}$
and $\{v_{\g,j} : \g \in \G\text{ and }j \in F'_\g\}$ in $E$, and elements $\{m_i : i \in
F\}$ and $\{n_{\g,j} : \g \in \G\text{ and }j \in F'_\g\}$ such that
\[
x=\sum_{i\in F}u_{i}p_{\a}m_{i}=\sum_{\g\in \G,\g\not = \a} \sum_{j\in F'_\g}v_{\g, j}p_{\g}n_{\g, j},
\]
Fix $e\in E$ such that $eu_{i}=u_{i}=u_{i}e$ for all $i\in F$. Using that the $u_{i}$ are
homogeneous elements of trivial degree at the second equality, we have
\begin{equation*}
ep_{\a}x
    =\sum_{i\in F}(ep_{\a}u_{i}p_{\a})m_{i}
    =\sum_{i\in F} eu_{i}p_{\a}m_{i}
    =x.
\end{equation*}
We also have
\begin{equation*}
    ep_{\a}x  = \sum_{\g\in \G\setminus\{\a\}} \sum_{j\in F'_\g} e p_{\a} v_{\g,j}p_{\g}n_{\g,j} =0.
\end{equation*}
Hence $x=0$.

For $r\in A_{\g}$ and $m\in M'_{\a}$, define $rm := rp_{\a}m$. This determines a left
$A$-action on $M_\alpha'$. For $u\in E$ satisfying $ur=r=ru$, we have
\begin{equation*}
    up_{\g\a}rm=(up_{\g\a}rp_{\a})m=urp_{\a}m=rm.
\end{equation*}
Hence $rm\in M'_{\g\a}$. One can easily check the associativity of the $A$-action. Using
Lemma~\ref{lemma-unital} we see that $M=M'$ as sets. We claim that $M'$ is a unital
$A$-module. For $m\in M_{\g}'$, we write $m=\sum_{u\in E'}up_{\g}m_{u}$, where
$E'\subseteq E$ is a finite set and $m_{u}\in M$. Since $u$ is a homogeneous idempotent,
\begin{equation*}
u(up_{\g}m_{u})
    =up_{\g}(up_{\g}m_{u})
    =up_{\g}m_{u}.
\end{equation*}
Thus $u(up_{\g}m_{u})=up_{\g}m_{u}\in AM'$ implies that $m\in AM'$ showing that $M'=AM'$.
We can therefore define $\phi : \operatorname{Obj}(A\#\G\-\Mod) \to\operatorname{Obj}(A\-\Gr) 
$ by $\phi(M)=M'$.

To define $\phi$ on morphisms, fix a morphism $f$ in $A\#\G\-\Mod$. For $m=\sum_{\g\in\G}m_{\g}\in M'$ such that $m_{\g}=\sum_{u\in F_{\g}}up_{\g}m_{u}$ with
$F_{\g}$ a finite subset of $E$, we define $f' : M' \to N'$ by
\begin{equation}\label{eq:graded}
    f'(m_{\g}) = f\Big(\sum_{u\in F_{\g}}up_{\g}m_{u}\Big) = \sum_{u\in F_{\g}}up_{\g}f(m_{u})=f(m)_{\g}.
\end{equation} To see that $f'$ is an
$A$-module homomorphism, fix $m\in M_{\g}'$ and $r\in A$. Since $f(m)\in M'_\g$, we have
\[
    f'(rm)=f(rp_\g m)=rp_\g f(m)=rf'(m). 
\] The definition~\eqref{eq:graded} shows that it
preserves the gradings. That is,  $f'$ is a $\G$-graded $A$-module homomorphism. So we can define $\phi$ on morphisms by $\phi(f) = f'$. It is routine to check
that $\phi$ is a functor.

Next we define a functor $\psi: A\-\Gr\xra A\#\G\-\Mod$ as follows. Let $N=\oplus_{\g\in
\G}N_\g$ be a $\G$-graded unital left $A$-module. Let $N''$ be a copy of $N$ as a group.
Fix $n \in N$, and write $n = \sum_{\g\in \G} n_\g$. Fix $r\in A$ and $\a\in\G$, and
define
\begin{equation*}
\label{doubleprimem}
(rp_{\a})n=rn_{\a}.
\end{equation*}
It is straightforward to check that this determines an associative left $A\#\G$-action on
$N''$. We claim that $N''$ is a unital $A\#\G$-module. To see this, fix $n\in N''$. Since
$AN=N$, we can express $n=\sum_{i=1}^{l}r_{i}n_{i}$, with the $n_i$ homogeneous in $N$
and the $r_i \in A$, and we can then write each $r_i$ as $r_{i}=\sum_{\b\in\G}r_{i, \b}$
as a sum of homogeneous elements $r_{i,\b} \in A_\b$. For any $\g\in\G$,
\[
n_{\g}=\sum_{i=1}^{l}r_{i, \b}{(n_{i})}_{\b^{-1}\g}=\sum_{i=1}^{l}(r_{i,
\b}p_{\b^{-1}\g})n_{i}\in (A\#\G)N''.
\]
So we can define $\psi :\operatorname{Obj}(A\-\Gr)\to \operatorname{Obj}(A\#\G\-\Mod)$
by $\psi(N) = N''$. Since $\psi(N) = N''$ is just a copy of $N$ as a module, we can
define $\psi$ on morphisms simply as the identity map; that is, if $f : M \to N$ is a homomorphism of graded $A$-modules, then for $m \in M$ we write $m''$ for the same element
regarded as an element of $M''$, and we have $\psi(f)(m'') = f(m)''$. Again, it is
straightforward to check that $\psi$ is a functor.

To prove that $\psi\circ \phi={\rm Id}_{A\#\G\-\Mod}$ and $\phi\circ \psi={\rm
Id}_{A\-\Gr}$, it suffices to show that $(M')''=M$ for $M\in A\#\G$-$\mathrm{Mod}$ and
$(N'')'=N$ for $N\in A$-$\mathrm{Gr}$; but this is straightforward from the definitions.

To prove the commutativity of the diagram in \eqref{inducedshift}, it suffices to show
that the $A\#\G$-actions on $(\psi\c \TT_{\a})(N) = N(\a)''$ and
$(\widetilde{\mathcal{S}}^\a \c \psi)(N) = N''(\a)$ coincide for any $N\in A\-\Gr$.
Take any $n\in N$ and $sp_{\b}\in A\#\G$ with $s\in A$ and $\b\in\G$. For $n\in N''(\a)$
and a typical spanning element $sp_\b$ of $A\#\G$, we have $(sp_{\b})n = (sp_{\b\a})n =
sn_{\b\a}$. On the other hand, for the same $n$ regarded as an element of $N''$, and the
same $sp_\b \in A\#\G$, we have $(sp_{\b})n = sn'_{\b} = sn_{\b\a}$. Since $N(\a)_{\b} =
N_{\b\a}$ by definition, this completes the proof.
\end{proof}

\section{The Steinberg algebra of the skew-product}\label{sectionthree}
In this section, we consider the skew-product of an ample groupoid $\mg$ carrying a
grading by a discrete group $\G$. We prove that the Steinberg algebra of the skew-product
is graded isomorphic to the smash product by $\G$ of the Steinberg algebra associated to
$\mg$. This result will be used in Section~\ref{sectionfive} to study the category of
graded modules over Leavitt path algebras and give a representation of the graded
finitely generated projective modules.

\subsection{Graded groupoids}

A groupoid is a small category in which every morphism is invertible. It can also be
viewed as a generalization of a group which has partial binary operation.  Let $\mg$ be a
groupoid. If $x\in\mg$, $d(x)=x^{-1}x$ is the \emph{domain} of $x$ and $r(x)=xx^{-1}$ is
its \emph{range}. The pair $(x,y)$ is composable if and only if $r(y)=d(x)$. The set
$\mg^{(0)}:=d(\mg)=r(\mg)$ is called the \emph{unit space} of $\mg$. Elements of
$\mg^{(0)}$ are units in the sense that $xd(x)=x$ and $r(x)x=x$ for all $x \in \mg$. For
$U,V\in\mg$, we define
\[
    UV=\big \{\a\b \mid \a\in U,\b\in V \text{ and } r(\b)=d(\a)\big\}.
\]

A topological groupoid is a groupoid endowed with a topology under which the inverse map
is continuous, and such that composition is continuous with respect to the relative
topology on $\mg^{(2)} := \{(\b,\g) \in \mg \times \mg : d(\b) = r(\g)\}$ inherited from
$\mg\times \mg$. An \emph{\'etale} groupoid is a topological groupoid $\mg$ such that the
domain map $d$ is a local homeomorphism. In this case, the range map $r$ is also a local
homeomorphism. An \emph{open bisection} of $\mg$ is an open subset $U\subseteq \mg$ such
that $d|_{U}$ and $r|_{U}$ are homeomorphisms onto an open subset of $\mg^{(0)}$. We say
that an \'etale groupoid $\mg$ is \emph{ample} if there is a basis consisting of compact
open bisections for its topology.

Let $\G$ be a discrete group and $\mg$ a topological groupoid. A $\G$-grading of $\mg$ is
a continuous function $c : \mg \to \G$ such that $c(\a)c(\b) = c(\a\b)$ for all $(\a,\b)
\in \mg^{(2)}$; such a function $c$ is called a \emph{cocycle} on $\mg$. In this paper,
we shall also refer to $c$ as the \emph{degree map} on $\mg$. Observe that $\mg$
decomposes as a topological disjoint union $\bigsqcup_{\g \in \mg} c^{-1}(\g)$ of subsets
satisfying $c^{-1}(\b)c^{-1}(\g) \subseteq c^{-1}(\b\g)$. We say that $\mg$ is
\emph{strongly graded} if $c^{-1}(\b)c^{-1}(\g) = c^{-1}(\b\g)$ for all $\b,\g$. For $\g
\in \G$, we say that $X\subseteq \mg$ is $\g$-graded if $X\subseteq c^{-1}(\g)$. We
always have $\mg^{(0)} \subseteq c^{-1}(\varepsilon)$, so $\mg^{(0)}$ is
$\varepsilon$-graded. We write $B^{\rm co}_{\g}(\mg)$ for the collection of all
$\g$-graded compact open bisections of $\mg$ and
\[
B_{*}^{\rm co}(\mg)=\bigcup_{\g\in\G} B^{\rm co}_{\g}(\mg).
\]

Throughout this note we only consider $\Gamma$-graded ample Hausdorff groupoids.

\subsection{Steinberg algebras}\label{subsetion32}

Steinberg algebras were introduced in~\cite{st} in the context of discrete inverse
semigroup algebras and independently in \cite{cfst} as a model for Leavitt path algebras.
We recall the notion of the Steinberg algebra as a universal algebra generated by certain
compact open subsets of an ample Hausdorff groupoid.

\begin{defi}
\label{defsteinberg} Let $\mg$ be a $\Gamma$-graded ample Hausdorff groupoid and
$B_{*}^{\rm co}(\mg)=\bigcup_{\g\in\G} B^{\rm co}_{\g}(\mg)$ the collection of all graded
compact open bisections. Given a commutative ring $R$ with identity, the \emph{Steinberg $R$-algebra} associated to
$\mg$, denoted $A_R(\mg)$, is the algebra generated by the set $\{t_B \mid B \in B^{\rm
co}_{*}(\mg)\}$ with coefficients in $R$, subject to
\begin{enumerate}
\item[(R1)] $t_{\emptyset}=0$;

\smallskip

\item[(R2)] $t_{B}t_{D}=t_{BD}$ for all $B, D\in B^{\rm co}_{*}(\mg)$; and

\smallskip

\item[(R3)] $t_{B}+t_{D}=t_{B\cup D}$ whenever $B$ and $D$ are disjoint elements of
    $B^{\rm co}_{\g}(\mg)$ for some $\g\in\G$ such that $B\cup D$ is a bisection.
\end{enumerate}
\end{defi}
Every element $f\in A_{R}(\mg)$ can be expressed as $f=\sum_{U\in F}a_{U}t_{U}$, where
$F$ is a finite subset of elements of $B_{*}^{\rm co}(\mg)$. It was proved
in~\cite[Proposition 2.3]{cm} (see also~\cite[Theorem 3.10]{cfst}) that the Steinberg
algebra defined above is isomorphic to the following construction:
\[
A_{R}(\mg) = \text{span}\{1_{U} : U \text{ is a compact open bisection of } \mg\},
\]
where $1_{U}:\mg \rightarrow R$ denotes the characteristic function on $U$. Equivalently,
if we give $R$ the discrete topology, then $A_R(\mg) = C_c(\mg, R)$, the space of
compactly supported continuous functions from $\mg$ to $R$. Addition is point-wise and
multiplication is given by convolution
$$(f*g)(\g) = \sum_{\{\a\b=\g\}}f(\a)g(\b).$$
It is useful to note that $$1_{U}*1_{V}=1_{UV}$$ for compact open bisections $U$ and $V$
(see \cite[Proposition 4.5(3)]{st}) and the isomorphism between the two constructions is
given by $t_U \mapsto 1_U$ on the generators. By \cite[Lemma 2.2]{cm} and \cite[Lemma
3.5]{cfst}, every element $f\in A_{R}(\mg)$ can be expressed as
\begin{equation}
\label{express}
f=\sum_{U\in F}a_{U}1_{U},
\end{equation}
where $F$ is a finite subset of mutually disjoint elements of $B_{*}^{\rm co}(\mg)$.

Recall from \cite[Lemma 3.1]{cs} that if $c : \mg \to \G$ is a continuous $1$-cocycle
into a discrete group $\G$, then the Steinberg algebra $A_R(\mg)$ is a $\Gamma$-graded
algebra with homogeneous components
\[
    A_{R}(\mg)_{\g} = \{f\in A_{R}(\mg)\mid \supp(f)\subseteq c^{-1}(\g)\}.
\]
The family of all idempotent elements of $A_{R}(\mg^{(0)})$ is a set of local units for
$A_{R}(\mg)$ (\cite[Lemma 2.6]{cep}). Here, $A_{R}(\mg^{(0)})\subseteq A_{R}(\mg)$ is a
subalgebra. Since $\mg^{(0)}\subseteq c^{-1}(\varepsilon)$ is trivially graded,
$A_{R}(\mg)$ is a graded algebra with graded local units. Note that any ample Hausdorff
groupoid admits the trivial cocycle from $\mg$ to the trivial group $\{\varepsilon\}$,
which gives rise to a trivial grading on $A_{R}(\mg)$.

\subsection{Skew-products} \label{subsection22}
Let $\mg$ be an ample Hausdorff groupoid, $\G$ a discrete group, and $c : \mg \to \G$ a
continuous cocycle. Then $\mg$ admits a basis $\mathcal{B}$ of compact open bisections.
Replacing $\mathcal{B}$ with $\mathcal{B}' = \{U \cap c^{-1}(\g) \mid U \in \mathcal{B},
\g \in \G\}$, we obtain a basis of compact open homogeneous bisections.

To a $\G$-graded groupoid $\mg$ one can associate a groupoid called the skew-product of
$\mg$ by $\G$. The aim of this section is to relate the Steinberg algebra of the
skew-product groupoid to the Steinberg algebra of $\mg$. We recall the notion of
skew-product of a groupoid (see~\cite[Definition 1.6]{re}).

\begin{defi} \label{defsp}
Let $\mg$ be an ample Hausdorff groupoid, $\G$ a discrete group and $c : \mg \to \G$ a
continuous cocycle. The \emph{skew-product} of $\mg$ by $\G$ is the groupoid $\mg
\times_c \G$ such that $(x,\a)$ and $(y, \b)$ are composable if $x$ and $y$ are
composable and $\a=c(y)\b$. The composition is then given by $\big (x,
c(y)\b\big)\big(y,\b\big)=(xy, \b)$ with the inverse $(x, \a)^{-1}=(x^{-1},c(x)\a)$.
\end{defi}

Note that our convention for the composition of the skew-product here is slightly
different from that in \cite[Definition 1.6]{re}. The two determine isomorphic groupoids,
but when we establish the isomorphism of Theorem~\ref{isothm}, the composition formula
given here will be more obviously compatible with the multiplication in the smash
product.

\begin{lem}\label{ample}
Let $\mg$ be a $\Gamma$-graded ample groupoid. Then  the skew-product $\mg\times_c \G$ is
a $\Gamma$-graded ample groupoid under the product topology on $\mg \times \G$ and with
degree map $\tilde{c}(x,\gamma) := c(x)$.
\end{lem}
\begin{proof}
We can directly check  that under the product topology on  $\mg \times \G$, the inverse
and composition of the skew-product  $\mg\times_c \G$ are continuous making it a
topological groupoid. Since the domain map $d: \mg\xra \mg^{(0)}$ is a local
homeomorphism, the domain map (also denoted $d$) from $\mg\times_c \G$ to
$\mg^{(0)}\times \G$ is $d \times \operatorname{id}_{\G}$ so restricts to a homeomorphism
on $U \times \G$ for any set $U$ on which $d$ is a homeomorphism. So $d : \mg\times_c \G
\to (\mg\times_c \G)^{(0)}$ is a local homeomorphism. Since the inverse map is clearly a
homeomorphism, it follows that the range map is also a local homeomorphism.

If $\mathcal{B}$ is a basis of compact open bisections for $\mg$, then $\{B \times \{\g\}
\mid B \in \mathcal{B}\text{ and }\g \in \G\}$ is a basis of compact open bisections for
the topology on $\mg\times_c \G$. Since composition on $\mg \times_c \G$ agrees with
composition in $\mg$ in the first coordinate, it is clear that $\tilde{c}$ is a
continuous cocycle.
\end{proof}

The Steinberg algebra $A_{R}(\mg\times_c \G)$ associated to $\mg\times_c \G$ is a
$\G$-graded algebra, with homogeneous components
\[
    A_{R}(\mg\times_c \G)_{\g} = \big\{f\in A_{R}(\mg\times_c \G) \mid \supp(f)\subseteq c^{-1}(\gamma)\times \Gamma\big\}
\]
for $\g\in\G$.

We are in a position to state the main result of this section.

\begin{thm} \label{isothm}
Let $\mg$ be a $\Gamma$-graded ample, Hausdorff groupoid and $R$ a unital commutative
ring.  Then there is an isomorphism of $\G$-graded algebras $A_{R}(\mg\times_c\G)\cong
A_{R}(\mg)\#\G,$ assigning $1_{U\times \{\alpha\}}$ to $1_U p_\a$ for each compact open
bisection $U$ of $\mg$ and $\a\in \G$.
\end{thm}
\begin{proof}
We first define a representation $\{t_{U} \mid  U\in B_{*}^{\rm co}(\mg\times_c \G)\}$ in
the algebra $A_{R}(\mg)\#\G$ (see Definition~\ref{defsteinberg}). If $U$ is a graded
compact open bisection of $\mg\times_c \G$, say $U \subseteq \tilde{c}^{-1}(\alpha)$, then for
each $\g\in\G$, the set $U \cap \mg \times \{\g\}$ is a compact open bisection. Since
these are mutually disjoint and $U$ is compact, there are finitely many (distinct) $\g_1,
\dots, \g_l \in \G$ such that $U = \bigsqcup^l_{i=1} U \cap \mg \times \{g_i\}$. Each $U
\cap \mg \times \{g_i\}$ has the form $U_i \times \{\g_i\}$ where $U_i \subseteq U$ is
compact open. The $U_i$ have mutually disjoint sources because the domain map on $\mg
\times_c \G$ is just $d \times \operatorname{id}$, and $U$ is a bisection. So each $U_i
\in B^{\rm co}_\alpha(\mg)$, and $U = \bigsqcup^l_{i=1} U_i \times \{\g_i\}$. Using this
decomposition, we define
\[
    t_{U}=\sum^l_{i=1} 1_{U_{i}}p_{\g_i}.
\]
We show that these elements $t_U$ satisfy (R1)--(R3). Certainly if $U=\emptyset$, then
$t_{U}=0$, giving~(R1). For~(R2), take $V \in B^{\rm co}_\b(\mg \times_c \G)$, and
decompose $V = \bigcup^m_{j=1} V_j \times \{\g'_j\}$ as above. Then
\begin{equation} \label{eqone}
\begin{split}
t_{U}t_{V}
    &=\sum_{i=1}^{l}1_{U_{i}}p_{\g_{i}}\sum_{j=1}^{m}1_{V_{j}}p_{\g'_{j}}
    =\sum_{i=1}^{l}\sum_{j=1}^{m}1_{U_{i}}p_{\g_{i}}1_{V_{j}}p_{\g'_{j}}\\
    &=\sum_{j=1}^{m}\sum_{i=1}^{l}1_{U_{i}}p_{\g_{i}}1_{V_{j}}p_{\g'_{j}}
    =\sum_{j=1}^{m}\sum_{\{1\leq i\leq l \mid \g_{i}(\g'_{j})^{-1}=\b\}}1_{U_{i}V_{j}}p_{\g'_{j}}.
\end{split}
\end{equation}

On the other hand, by the composition of the skew-product $\mg\times_c \Gamma$, we have
\begin{equation*}
\begin{split}
UV  &=\bigcup_{i=1}^{l}\bigcup_{j=1}^{m}U_{i}\times\{\g_{i}\}\cdot V_{j}\times\{\g'_{j}\}\\
    &=\bigcup_{j=1}^{m}\bigcup_{i=1}^{l}U_{i}\times\{\g_{i}\}\cdot V_{j}\times\{\g'_{j}\}
    =\bigcup_{j=1}^{m}\bigcup_{\{1\leq i\leq l \mid \g_{i} = \b\g'_{j}\}}U_{i}V_{j}\times \{\g'_{j}\}.
\end{split}
\end{equation*}
For each $1\leq j\leq m$, there exists at most one $1\leq i\leq l$ such that
$\g_{i}=\b\g'_{j}$ and $U_{i}V_{j}\in B^{\rm co}_{\a\b}(\mg)$. It follows that
$t_{UV}=\sum_{j=1}^{s}\sum_{\{1\leq i\leq l \mid
\g_{i}(\g'_{j})^{-1}=\b\}}1_{U_{i}V_{j}}p_{\g'_{j}}$. Comparing this with~\eqref{eqone},
we obtain $t_{U}t_{V}=t_{UV}$.

To check~(R3), suppose that $U$ and $V$ are disjoint elements of $B^{\rm
co}_{\omega}(\mg\times_c \G)$ for some $\omega \in \G$ such that $U\cup V$ is a bisection
of $\mg\times_c \Gamma$. Write them as $U=\bigcup_{i=1}^{l} U_{i}\times \{\g_{i}\}$ and
$V=\bigcup_{j=1}^{m} V_{i}\times \{\g'_{j}\}$ as above. We have
$t_U+t_V=\sum_{i=1}^{l}1_{U_{i}}p_{\g_{i}}+\sum_{j=1}^{m}1_{V_{j}}p_{\g'_{j}}$. On the
other hand $U\cup V = (\bigcup_{i=1}^{l} U_{i}\times \{\g_{i}\})\bigcup
(\bigcup_{j=1}^{m} V_{i}\times \{\g'_{j}\})$. If $\g_i=\g'_j$, then $U_i\times \{\g_{i}\}
\cup V_j\times \{\g'_{j}\} = (U_i\cup V_j) \times \{\g_{i}\}$. Since $U$ and $V$ are
disjoint and $U\cup V$ is a bisection, we deduce that $r(U_i) \cap r(V_j) = \emptyset =
d(U_i) \cap d(V_j)$ so that $U_i\cup V_j$ is a bisection.  So
\[
t_{U_i\times \{\g_{i}\} \cup V_j\times \{\g'_{j}\}}
    = t_{U_i\cup V_j \times \{\g_{i}\}}
    = 1_{U_i\cup V_j} p_{\g_i}
    = 1_{U_i} p_{\g_i} + 1_{V_j}p_{\g_i}
    = 1_{U_i}p_{\g_i} + 1_{V_j}p_{\g'_j}.
\]
This shows that after combining pairs where $\g_i = \g'_j$ as above, we obtain
$t_{U}+t_{V}=t_{U\cup V}$.

By the universality of Steinberg algebras, we have an $R$-homomorphism,
\begin{equation*}
\phi:A_{R}(\mg\times_c\G) \longrightarrow A_{R}(\mg)\#\G
\end{equation*}
such that $\phi(1_{U\times \{\a\}})=1_{U}p_{\a}$ for each compact open bisection $U$ of $\mg$ and $\a\in\G$.  From the
definition of $\phi$, it is evident that $\phi$ preserves the grading. Hence, $\phi$ is a
homomorphism of $\G$-graded algebras.

Next we prove that $\phi$ is an isomorphism. For any element $ap_{\g}\in A_{R}(\mg)\#\G$
with $a\in A_{R}(\mg)$ and $\g\in\G$, there is a finite index set $T$, elements $\{r_i
\mid i \in T\}$ of $R$, and compact open bisections $K_{i}\in B^{\rm co}_{*}(\mg)$ such
that
\begin{equation*}
ap_{\g}
    = \sum_{i\in T}r_{i}1_{K_{i}}p_{\g}
    = \sum_{i\in T}r_{i}\phi(1_{K_{i}\times\{\g\}})\in \mathrm{Im} \phi.
\end{equation*}
So $\phi$ is surjective. It remains to prove that $\phi$ is injective. Take an element
$x\in A_{R}(\mg\times_c \G)$ such that $\phi(x)=0$. Since $\phi$ is graded, we can assume
that $x$ is homogeneous, say $x \in A_R(\mg \times_c \G)_\gamma$. By \eqref{express},
there is a finite set $F$, mutually disjoint $B_{i}\in B^{\rm co}_{\gamma}(\mg\times_c
\G)$ indexed by $i \in F$ and coefficients $r_{i}\in R$ indexed by $i \in F$ such that
\[
x=\sum_{i\in F}r_{i}1_{B_{i}}.
\]
For each $B_{i}$, we write $B_{i}=\bigcup_{k\in F_{i}} B_{ik}\times \{\delta_{ik}\}$ such
that $F_{i}$ is a finite set and the $\delta_{ik}$ indexed by $k \in F_i$ are distinct.
Set
\[
\Delta=\{\delta_{ik} \mid i\in F, k\in F_{i}\}.
\]
For each $\delta \in \Delta$, let $F_\delta \subseteq F$ be the collection $F_\delta
= \big\{i \in F : \delta \in \{\delta_{ik} : k \in F_i\}\big\}$. Then
\[
\phi(x)
    = \sum_{i\in F}r_{i}\phi(1_{B_{i}})
    = \sum_{i\in F}\sum_{k\in F_{i}}r_{i}1_{B_{ik}}p_{\delta_{ik}}
    = \sum_{\delta\in \Delta}\sum_{i\in F_{\delta}} r_{i}1_{B_{i, k(\delta)}}p_{\delta}
    = 0.
\]
For any $\delta\in\Delta$, we obtain $\sum_{i\in F_{\delta}}r_{i}1_{B_{i, k(\delta)}}=0$.
Since the $B_{i}$ are mutually disjoint, for any element $g\in\mg$, we have
\begin{equation*}
\Big(\sum_{i\in F_{\delta}}r_{i}1_{B_{i, k(\delta)}}\Big)(g)
    =\begin{cases}
        r_{i}, & \text{if~} g\in B_{i, k(\delta)} \text{~for some~} i\in F_{\delta};\\
        0, & \text{otherwise}.
    \end{cases}
\end{equation*}
Then $r_{i}=0$ for any $i\in F_{\delta}$, giving $x=0$.
\end{proof}

\subsection{\texorpdfstring{$C^*$}{C*}-algebras and crossed-products}\label{aidancalgebra}

In the groupoid-$C^*$-algebra literature, it is well-known that if $\mg$ is a
$\Gamma$-graded groupoid, and $\Gamma$ is abelian, then the $C^*$-algebra $C^*(\mg \times
\Gamma)$ of the skew-product groupoid is isomorphic to the crossed  product $C^*$-algebra
$C^*(\mg) \times_{\alpha^c} \widehat{\Gamma}$, where $\alpha^c$ is the action of the
Pontryagin dual $\widehat{\Gamma}$ such that $\alpha^c_\chi(f)(g) = \chi(c(g))f(g)$ for
$f \in C_c(\mg)$, $\chi \in \widehat{\Gamma}$, and $g \in \mg$. This extends to
nonabelian $\Gamma$ via the theory of $C^*$-algebraic coactions.

In this subsection, we reconcile this result with Theorem~\ref{isothm} by showing that
there is a natural embedding of $A_\mathbb{C}(\mg)\#\Gamma$ into $C^*(\mg)
\times_{\alpha^c} \widehat{\Gamma}$ when $\Gamma$ is abelian.

\begin{lem}\label{lem:smashprod embedding}
Suppose that $\Gamma$ is a discrete abelian group and that $\mg$ is a $\Gamma$-graded
groupoid with grading cocycle $c : \mg \to \Gamma$. For $a \in A_{\mathbb{C}}(\mg)$ and
$\g \in \G$, define $a \cdot \hat\g \in C(\widehat{\G},  C^*(\mg)) \subseteq C^*(\mg)
\times_{\alpha^c} \G$ by
\[
\big(a \cdot \hat{\gamma}\big)(\chi) = \chi(\gamma)a.
\]
Then there is a homomorphism $A_\mathbb{C}(\mg)\#\Gamma \hookrightarrow C^*(\mg)
\times_{\alpha^c} \widehat{\Gamma}$ that carries $ap_\gamma$ to $a \cdot \hat{\gamma}$.
\end{lem}
\begin{proof}
The multiplication in the crossed-product $C^*$-algebra is given on elements of
$C(\widehat{\Gamma}, C^*(\mg))$ by $(F*G)(\chi) = \int_{\widehat{\Gamma}} F(\rho)
\alpha^c_{\rho}(G(\rho^{-1}\chi))\,d\mu(\rho)$, where $\mu$ is Haar measure on
$\widehat{\Gamma}$.

The action of $\widehat{\Gamma}$ induces a $\Gamma$-grading of $C^*(\mg)
\times_{\alpha^c} \widehat{\Gamma}$ such that for $a \in C^*(\mg) \times_{\alpha^c}
\widehat{\Gamma}$ and $\gamma \in \Gamma$, the corresponding homogeneous component
$a_\gamma$ of $a$ is given by
\[
a_\gamma = \int_{\widehat{\Gamma}} \overline{\chi(\gamma)} \alpha^c_\chi(a)\,d\mu(\chi).
\]

There is certainly a linear map $i : A_\mathbb{C}(\mg)\#\Gamma \to C^*(\mg)
\times_{\alpha^c} \widehat{\Gamma}$ satisfying $i(ap_\gamma) = a\cdot\hat{\gamma}$; we
just have to check that it is multiplicative. For this, fix $a,b \in A_\mathbb{C}(\mg)$
and $\gamma,\beta \in \Gamma$ and $\chi \in \widehat{\Gamma}$, and calculate
\begin{align*}
\big(i(ap_\gamma)i(bp_\beta)\big)(\chi)
    &= \int_{\widehat{\Gamma}} i(ap_\gamma)(\rho) \alpha^c_{\rho}(i(bp_\beta)(\rho^{-1}\chi))\,d\mu(\rho)
     = \int_{\widehat{\Gamma}} a\cdot \hat{\gamma}(\rho) \alpha^c_{\rho}(b\hat{\beta}(\rho^{-1}\chi))\,d\mu(\rho)\\
    &= \int_{\widehat{\Gamma}} \rho(\gamma) a (\rho^{-1}\chi)(\beta) \alpha^c_{\rho}(b)\,d\mu(\rho)
     = \chi(\beta) a \int_{\widehat{\Gamma}} \overline{\rho(\gamma^{-1}\beta)} \alpha^c_{\rho}(b)\,d\mu(\rho)\\
    &= \chi(\beta) a b_{\gamma^{-1}\beta}
     = (a b_{\gamma^{-1}\beta})\cdot\hat{\beta}
     = i(ab_{\gamma^{-1}\beta}p_\beta)
     = i(ap_\gamma bp_\beta).
\end{align*}
So $i$ is multiplicative as required.
\end{proof}

\section{Non-stable graded \texorpdfstring{$K$}{K}-theory}\label{sectiondrdr}
For a unital ring $A$, we denote by $\VV(A)$ the abelian monoid of isomorphism classes of
finitely generated projective left $A$-modules under direct sum. In general for an
abelian monoid $M$ and elements $x,y \in M$, we write $x\leq y$ if $y=x+z$ for some $z\in
M$. An element $d\in M$ is called \emph{distinguished} (or an \emph{order unit}) if for
any $x\in M$, we have $x\leq nd$ for some $n\in \mathbb N$. A monoid is called
\emph{conical}, if $x+y=0$ implies $x=y=0$. Clearly $\VV(A)$ is conical with a
distinguished element $[A]$. For a finitely generated conical abelian monoid $M$
containing a distinguished element $d$, Bergman constructed a ``universal'' $K$-algebra
$B$ (here $K$ is a field) for which there is an isomorphism $\phi:\VV(B)\rightarrow M$,
such that $\phi([B])\rightarrow d$ (\cite[Theorem~6.2]{bergman74}).

For a (finite) directed graph $E$, one defines an abelian monoid $M_E$ generated by the
vertices, identifying a vertex with the sum of vertices connected to it by edges
(see~\S\ref{stlib}). The Bergman universal algebra associated to this monoid (with the
sum of vertices as a distinguished element) is the Leavitt path algebra $L_K(E)$
associated  to the graph $E$, i.e., $\VV(L_K(E))\cong M_E$.  Leavitt path algebras of
directed graphs have been studied intensively since their introduction~\cite{ap,amp}. The
classification of such algebras is still a major open topic and one would like to find a
complete invariant for such algebras. Due to the success of $K$-theory in the
classification of graph $C^*$-algebras~\cite{phillips}, one would hope that the
Grothendieck group $K_0$ with relevant ingredients might act as a complete invariant for
Leavitt path algebras; particularly since $K_0(L_K(E))$ is the group completion of
$\VV(L_K(E))$. However, unless the graph consists of only cycles with no exit,
$\VV(L_K(E))$ is not a cancellative monoid (Lemma~\ref{iffcancellative}) and thus
$\VV(L_K(E))\rightarrow K_0(L_K(E))$ is not injective, reflecting that $K_0$ might not
capture all the properties of $L_K(E)$.

For a graded ring $A$ one can consider the abelian monoid of isomorphism classes of
graded finitely generated projective modules denoted by $\VV^{\gr}(A)$. Since a Leavitt
path algebra has a canonical $\mathbb Z$-graded structure, one can consider
$\VV^{\gr}(L_K(E))$. One of the main aims of this paper is to show that the graded monoid
carries substantial information about the algebra.

In Sections~\ref{sectionfive} and \ref{sectionsix} we will use the results on smash
products obtained in Section~\ref{sectionthree} to study the graded monoid of Leavitt
path algebras and Kumjian--Pask algebras. In this section we collect the facts we need on
the graded monoid of a graded ring with graded local units.

\subsection{The monoid of a graded ring with graded local units}\label{kjhgfrty}

For a ring $A$ with unit, the monoid $\mathcal{V}(A)$ is defined as the set of
isomorphism classes $[P]$ of finitely generated projective $A$-modules $P$, with addition
given by $[P] + [Q] = [P\oplus Q]$.

For a non-unital ring $A$, we consider a unital ring $\widetilde{A}$ containing $A$ as a
two-sided ideal and define
\begin{equation}\label{monoididempotent0}
\VV(A)=\{[P] \mid  P \text{~is a finitely generated projective~} \widetilde{A}\-\text{module and } P=AP\}.
\end{equation}
This construction does not depend on the choice of $\widetilde{A}$, as can be seen from
the following alternative description: $\VV(A)$ is the set of equivalence classes of
idempotents in $M_{\infty}(A)$, where $e\sim f$ in $M_{\infty}(A)$ if and only if there
are $x, y\in M_{\infty}(A)$ such that $e=xy$ and $f=yx$ (\cite[pp. 296]{mm}).

When $A$ has local units,
\begin{equation}\label{monoididempotent}
\VV(A)=\{[P] \mid  P \text{~is a finitely generated projective~} A\-\text{module in~} A\-\Mod\}.
\end{equation}

To see this, recall that the \emph{unitisation ring} $\widetilde{A}$ of a ring $A$ is a
copy of $\mathbb{Z}\times A$ with componentwise addition, and with multiplication given
by
\[
    (n, a)(m, b)=(nm, ma+nb+ab)\quad\text{for all $n,m\in \mathbb{Z}$ and $a, b\in A$.}
\]
The forgetful functor provides a category isomorphism from $\widetilde{A}\-\Mod$ to the
category of arbitrary left $A$-modules \cite[Proposition 8.29B]{faith}. Any $A$-module
$N$ can be viewed as a $\widetilde{A}$-module via $(m,b)x=mx+bx$ for $(m,
b)\in\widetilde{A}$ and $x\in N$. By \cite[Lemma 10.2]{ag}, the projective objects in
$A\-\Mod$ are precisely those which are projective as $\widetilde{A}$-modules; that is,
the projective $\widetilde{A}$-modules $P$ such that $AP=P$. Any finitely generated
$\widetilde{A}$-module $M$ with $AM=M$ is a finitely generated $A$-module. In fact,
suppose that $M$ is generated as an $\widetilde{A}$-module by $x_{1}, \cdots, x_{n}$.
Since $AM=M$, each $x_i$ can be written as $x_{i}=\sum_{j=1}^{t_{i}}b_{j}x_{ij}$ for some
$b_{j}\in A$ and $x_{ij}\in M$. Now any $m\in M$ can be written
\[
m = \sum_{i=1}^{n}a_{i}x_{i} = \sum_{i=1}^{n}\sum_{j=1}^{t_{i}}a_{i}b_{j}x_{ij}
\]
So $\{x_{ij} \mid 1 \le i \le n\text{ and }1 \le j \le t_{i}\}$ generates $M$ as an
$A$-module. Clearly any finitely generated $A$-module is a finitely generated
$\widetilde{A}$-module. So the definitions of $\VV(A)$ in
(\ref{monoididempotent0})~and~(\ref{monoididempotent}) coincide.

We need a graded version of (\ref{monoididempotent}) as this presentation will be used to
study the monoid associated to the Leavitt path algebras of arbitrary graphs.

Recall that for a group $\G$ and a $\G$-graded ring $A$ with unit, the monoid
$\VV^{\gr}(A)$ consists of isomorphism classes $[P]$ of graded finitely generated
projective $A$-modules with the direct sum $[P] + [Q] = [P\oplus Q]$ as the addition
operation.

For a non-unital graded ring $A$, a similar construction as in (\ref{monoididempotent0})
can be carried over to the graded setting (see~\cite[\S3.5]{haz}). Let $\widetilde{A}$ be
a $\G$-graded ring with identity such that $A$ is a graded two-sided ideal of $A$. For
example, consider $\widetilde{A}=\mathbb{Z}\times A$. Then $\widetilde{A}$ is $\G$-graded
with
\[
\widetilde{A}_{0}=\mathbb{Z}\times A_{0},\qquad\text{ and }\qquad
    \widetilde{A}_{\g} = 0\times A_{\g} \text{ for } \g\neq 0.
\]
Define
\begin{equation}\label{grlocal0}
\VV^{\gr}(A)=\{[P] \mid P \text{~is a graded finitely generated projective~} \widetilde{A}\-\text{module and~} AP=P\},
\end{equation}
where $[P]$ is the class of graded $\widetilde{A}$-modules, graded isomorphic to $P$, and
addition is defined via direct sum. Then $\VV^{\gr}(A)$ is isomorphic to the monoid of
equivalence classes of graded idempotent matrices over $A$ \cite[pp.~146]{haz}.

Let $A$ be a $\G$-graded ring with graded local units. We will show that
\begin{equation} \label{grlocal}
\VV^{\gr}(A)=\{[P] \mid  P \text{~is a graded finitely generated projective~} A\-\text{module in~} A\-\Gr\}.
\end{equation}

For this we need to relate the graded projective modules to modules which are projective.
A graded $A$-module $P$ in $A\-\Gr$ is called a graded projective $A$-module if for any
epimorphism $\pi:M\xra N$ of graded $A$-modules in $A\-\Gr$ and any morphism $f:P\xra N$
of graded $A$-modules in $A\-\Gr$, there exists a morphism $h:P\xra M$ of graded
$A$-modules such that $\pi\c h=f$.

In the case of unital rings, a module is graded projective if and only if it is graded
and projective~\cite[Prop.~1.2.15]{haz}. We need a similar statement in the setting of
rings with local units.

\begin{lem}\label{grprojnon}
Let $A$ be a $\G$-graded ring with graded local units and $P$ a graded unital left
$A$-module. Then $P$ is a graded projective left $A$-module in $A\-\Gr$ if and only if
$P$ is a graded left $A$-module which is projective in $A\-\Mod$.
\end{lem}
\begin{proof}
First suppose that $P$ is a graded projective $A$-module in $A\-\Gr$. It suffices to
prove that $P$ is  projective in $A\-\Mod$. For any homogeneous element $p\in P$ of
degree $\delta_{p}$, there exists a homogeneous idempotent $e_{p}\in A$ such that
$e_{p}p=p$. Let $\bigoplus_{p\in P^{h}}Ae_{p}(-\delta_{p})$ be the direct sum of graded
$A$-modules where ${\rm deg}(e_{p})=\delta_{p}$ and $P^{h}$ is the set of homogeneous
elements of $P$. Then there exists a surjective graded $A$-module homomorphism
\[
f: \bigoplus_{p\in P^{h}}Ae_{p}(-\delta_{p})\longrightarrow P
\]
such that $f(ae_{p})=ae_{p}p=ap$ for $a\in Ae_{p}$. Since $P$ is graded projective, there
exists a graded $A$-module homomorphism $g:P\xra \bigoplus_{p\in
P^{h}}Ae_{p}(-\delta_{p})$ such that $fg={\rm Id}_{P}$. Forgetting the grading, $P$ is a
direct summand of $\bigoplus_{p\in P^{h}}Ae_{p}$ as an $A$-module. By \cite[49.2(3)]{wi},
$\bigoplus_{p\in P^{h}}Ae_{p}$ is projective in $A\-\Mod$. So $P$ is projective in
$A\-\Mod$.

Conversely, suppose that $P$ is a graded and projective $A$-module. Let $\pi:M\xra N$ be
an epimorphism of graded $A$-modules in $A\-\Gr$ and $f:P\xra N$ a morphism of graded
$A$-modules in $A\-\Gr$. We first claim that any epimorphism $\pi:M\xra N$ of graded
$A$-modules in $A\-\Gr$ is surjective. To prove the claim, write $A^{h}$ for the set of
all homogeneous elements of $A$. Let $X = \{x\in N \mid A^{h}x\subseteq \pi(M)\}\subseteq
N$ (cf. \cite[\S 5.3]{go}). Then $X$ is a graded submodule of $N$. We denote by $q:N\xra
N/X$ the quotient map. Then $q\c\pi=0$. Hence, $q=0$, giving $N=X$. It follows that
$N=\pi(M)$. So the epimorphism $\pi:M\xra N$ of graded $A$-modules in $A\-\Gr$ is
surjective. Forgetting the grading, $\pi:M\xra N$ is a surjective morphism of $A$-modules
in $A\-\Mod$. Since $P$ is projective in $A\-\Mod$, there exists $h:P\xra M$ such that
$\pi\c h=f$. By  \cite[Lemma 1.2.14]{haz}, there exists a morphism $h':P\xra M$ of graded
$A$-modules such that $\pi\c h'=f$. Thus, $P$ is a graded projective left $A$-module in
$A\-\Gr$.
\end{proof}

Thus for a $\G$-graded ring $A$ with graded local units, combining Lemma~\ref{grprojnon}
with \cite[Lemma 10.2]{ag} (i.e., projective objects in $A\-\Mod$ are precisely those
that are projective as $\widetilde{A}$-modules),  $P$ is a graded finitely generated
projective $\widetilde{A}$-module with $AP=P$ if and only if $P$ is a finitely generated
$A$-module which is graded  projective in $A\-\Gr$. This shows that the definitions of
$\VV^{\gr}(A)$ by (\ref{grlocal0}) and (\ref{grlocal}) coincide.

\section{Application: Leavitt path algebras}\label{sectionfive}
In this section we study the monoid $\VV^{\gr}(L_K(E))$ of the Leavitt path algebra of a
graph $E$ (\ref{grlocal}). Using the results on smash products of Steinberg algebras
obtained in Section~\ref{sectionthree}, we give a presentation for this monoid in line
with $M_E$ (see~\S\ref{stlib}). Using this presentation we show that $\VV^{\gr}(L_K(E))$
is a cancellative monoid and thus the natural map $\VV^{\gr}(L_K(E)) \rightarrow
K_0^{\gr}(L_K(E))$ is injective (Corollary~\ref{corinj}). It follows that there is a
lattice correspondence between  the graded ideals of $L_K(E)$ and the graded ordered ideals
of $K_0^{\gr}(L_K(E))$ (Theorem~\ref{statelib2}). This is further evidence for the
conjecture that the graded Grothendieck group $K_0^{\gr}$ may be a complete invariant for
Leavitt path algebras~\cite{haz2013}.

\subsection{Leavitt path algebras modelled as Steinberg algebras}\label{subsection41}
We briefly recall the definition of Leavitt path algebras and establish notation. We
follow the conventions used in the literature of this topic (in particular the paths are
written from left to right).

A directed graph $E$ is a tuple $(E^{0}, E^{1}, r, s)$, where $E^{0}$ and $E^{1}$ are
sets and $r,s$ are maps from $E^1$ to $E^0$. We think of each $e \in E^1$ as an arrow
pointing from $s(e)$ to $r(e)$. We use the convention that a (finite) path $p$ in $E$ is
a sequence $p=\a_{1}\a_{2}\cdots \a_{n}$ of edges $\a_{i}$ in $E$ such that
$r(\a_{i})=s(\a_{i+1})$ for $1\leq i\leq n-1$. We define $s(p) = s(\a_{1})$, and $r(p) =
r(\a_{n})$. If $s(p) = r(p)$, then $p$ is said to be closed. If $p$ is closed and
$s(\a_i) \neq s(\a_j)$ for $i\neq j$, then $p$ is called a cycle. An edge $\a$ is an exit
of a path $p=\a_1\cdots \a_{n}$ if there exists $i$ such that $s(\a)=s(\a_i)$ and
$\a\neq\a_i$. A graph $E$ is called acyclic if there is no closed path in $E$.

A directed graph $E$ is said to be \emph{row-finite} if for each vertex $u\in E^{0}$,
there are at most finitely many edges in $s^{-1}(u)$. A vertex $u$ for which $s^{-1}(u)$
is empty is called a \emph{sink}, whereas $u\in E^{0}$ is called an \emph{infinite
emitter} if $s^{-1}(u)$ is infinite. If $u\in E^{0}$ is neither a sink nor an infinite
emitter, then it is called a \emph{regular vertex}.

\begin{defi} \label{defleavitt} Let $E$ be a directed graph and $R$ a commutative ring with unit. The \emph{Leavitt path algebra} $L_{R}(E)$ of $E$ is the $R$-algebra generated by the set $\{v \mid v\in E^{0}\}\cup \{e \mid e\in E^{1}\}
\cup\{e^{*} \mid e\in E^{1}\}$ subject to the following relations:
\begin{itemize}
\item[(0)] $uv=\delta_{u, v}v$ for every $u, v\in E^{0}$;
\item[(1)] $s(e)e=er(e)=e$ for all $e\in E^{1}$;
\item[(2)] $r(e)e^{*}=e^{*}=e^{*}s(e)$ for all $e\in E^{1}$;
\item[(3)] $e^{*}f=\delta_{e, f}r(e)$ for all $e, f\in E^{1}$; and
\item[(4)] $v = \sum_{e\in s^{-1}(v)}ee^{*}$ for every regular vertex $v\in
    E^{0}$.
\end{itemize}
\end{defi}

Let $\G$ be a group with identity $\varepsilon$, and let $w :E^{1}\xra \G$ be a function.
Extend $w$ to vertices and ghost edges by defining  $w(v) = \varepsilon$ for $v\in E^{0}$
and $w(e^{*}) = w(e)^{-1}$ for $e\in E^{1}$. The relations in Definition~\ref{defleavitt}
are compatible with $w$, so there is a grading of $L_R(E)$ such that $e \in
L_R(E)_{w(e)}$ and $e^* \in L_R(E)_{w(e)^{-1}}$ for all $e \in E^1$, and $v \in
L_R(E)_\varepsilon$ for all $v \in E^0$. The set of all finite sums of distinct elements
of $E^{0}$ is a set of graded local units for $L_R(E)$ \cite[Lemma 1.6]{ap}. Furthermore,
$L_{R}(E)$ is unital if and only if $E^0$ is finite.

Leavitt path algebras associated to arbitrary graphs can be realised as Steinberg
algebras.  We recall from \cite[Example 2.1]{cs} the construction of the groupoid
$\mg_{E}$ from an arbitrary graph $E$, which was introduced in \cite{kprr} for row-finite
graphs and generalised to arbitrary graphs in \cite{pa}. We then realise the  Leavitt
path algebra $L_R(E)$ as the Steinberg algebra $A_R(\mg)$. This allows us to apply
Theorem~\ref{isothm} to the setting of Leavitt path algebras.

Let $E=(E^{0}, E^{1}, r, s)$ be a directed graph. We denote by $E^{\infty}$ the set of
infinite paths in $E$ and by $E^{*}$ the set of finite paths in $E$. Set
\[
X := E^{\infty}\cup  \{\mu\in E^{*}  \mid   r(\mu) \text{ is not a regular vertex}\}.
\]
Let
\[
\mg_{E} := \{(\a x,|\a|-|\b|, \b x) \mid   \a, \b\in E^{*}, x\in X, r(\a)=r(\b)=s(x)\}.
\]
We view each $(x, k, y) \in \mg_{E}$ as a morphism with range $x$ and source $y$. The
formulas $(x,k,y)(y,l,z)= (x,k + l,z)$ and $(x,k,y)^{-1}= (y,-k,x)$ define composition
and inverse maps on $\mg_{E}$ making it a groupoid with $\mg_{E}^{(0)}=\{(x, 0, x) \mid
x\in X\}$ which we identify with the set $X$.

Next, we describe a topology on $\mg_{E}$. For $\mu\in E^{*}$ define
\[
Z(\mu)= \{\mu x \mid x \in X, r(\mu)=s(x)\}\subseteq X.
\]
For $\mu\in E^{*}$ and a finite $F\subseteq s^{-1}(r(\mu))$, define
\[
Z(\mu\setminus F) = Z(\mu) \setminus \bigcup_{\a\in F} Z(\mu \a).
\]
The sets $Z(\mu\setminus F)$ constitute a basis of compact open sets for a locally
compact Hausdorff topology on $X=\mg_{E}^{(0)}$ (see \cite[Theorem 2.1]{we}).

For $\mu,\nu\in E^{*}$ with $r(\mu)=r(\nu)$, and for a finite $F\subseteq E^{*}$ such
that $r(\mu)=s(\a)$ for $\a\in F$, we define
\[
Z(\mu, \nu)=\{(\mu x, |\mu|-|\nu|, \nu x) \mid  x\in X, r(\mu)=s(x)\},
\]
and then
\[
Z((\mu, \nu)\setminus F) = Z(\mu, \nu) \setminus \bigcup_{\a\in F}Z(\mu\a, \nu\a).
\]
The sets $Z((\mu, \nu)\setminus F)$ constitute a basis of compact open bisections for a
topology under which $\mg_{E}$ is a Hausdorff ample groupoid. By \cite[Example~3.2]{cs},
the map
\begin{equation}
\label{assteinberg}
\pi_{E} : L_{R}(E) \longrightarrow A_{R}(\mg_{E})
\end{equation}
defined by $\pi_{E}(\mu\nu^{*}-\sum_{\a\in F}\mu\a\a^{*}\nu^{*}) =
1_{Z((\mu,\nu)\setminus F)}$ extends to an algebra isomorphism. We observe that the
isomorphism of algebras in \eqref{assteinberg} satisfies
\begin{equation}
\label{valuation}
\pi_{E}(v)=1_{Z(v)}, \quad\pi_{E}(e)=1_{Z(e, r(e))}, \quad\pi_{E}(e^{*})=1_{Z(r(e), e)},
\end{equation}
for each $v\in E^{0}$ and $e\in E^{1}$.

\subsection{Covering of a graph}\label{covergraph}
In this section we show that the smash product of a Leavitt path algebra is isomorphic to
the Leavitt path algebra of its covering graph. We briefly recall the concept of skew
product or covering of a graph (see~\cite[\S2]{gr} and \cite[Def.~2.1]{kp}).

Let $\G$ be a group and $w:E^{1}\xra \G$ a function. As in \cite[\S2]{gr}, the
\emph{covering graph} $\overline{E}$ of $E$ with respect to $w$ is given by
\begin{gather*}
    \overline{E}^{0} = \{v_{\a} \mid  v\in E^{0}\text{ and } \a\in\G\},\qquad
    \overline{E}^{1} = \{e_\a \mid e \in E^1\text{ and } \a \in \G\},\\
    s(e_\a) = s(e)_\a,\quad\text{ and } r(e_\a) = r(e)_{w(e)^{-1}\a}.
\end{gather*}

\begin{exm}\label{onetwo3}
Let $E$ be a graph and define $w:E^1 \rightarrow \Z$ by $w(e) = 1$ for all $e \in E^1$.
Then $\overline E$ (sometimes denoted $E\times_1 \mathbb Z$) is given by
\begin{gather*}
    \overline E^0 = \big\{v_n \mid v \in E^0 \text{ and } n \in \Z \big\},\qquad
    \overline E^1 = \big\{e_n \mid e\in E^1 \text{ and } n\in \Z \big\},\\
    s(e_n) = s(e)_n,\qquad\text{ and } r(e_n) = r(e)_{n-1}.
\end{gather*}

As examples, consider the following graphs
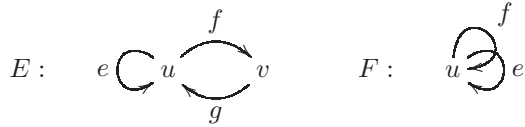
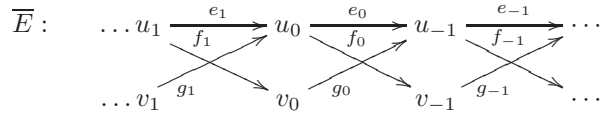
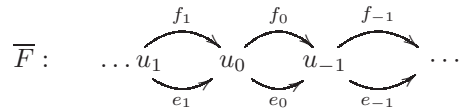
\begin{equation*}
{\def\labelstyle{\displaystyle}
E : \quad \,\, \xymatrix{
 u \ar@(lu,ld)_e\ar@/^0.9pc/[r]^f & v \ar@/^0.9pc/[l]^g
 }} \qquad \quad
{\def\labelstyle{\displaystyle}
F: \quad \,\, \xymatrix{
   u \ar@(ur,rd)^e  \ar@(u,r)^f
}}
\end{equation*}
Then
\begin{equation*}
\overline{E}: \quad \,\,\xymatrix@=15pt{
\dots  {u_{1}} \ar[rr]^-{e_1} \ar[drr]^(0.4){f_1} &&  {u_{0}} \ar[rr]^-{e_0} \ar[drr]^(0.4){f_0} && {u_{-1}}  \ar[rr]^-{e_{-1}} \ar[drr]^(0.4){f_{-1}} && \cdots\\
\dots {v_{1}}   \ar[urr]_(0.3){g_1} && {v_{0}} \ar[urr]_(0.3){g_0}  && {v_{-1}}  \ar[urr]_(0.3){g_{-1}}&& \cdots
}
\end{equation*}
and
\begin{equation*}
\overline{F}: \quad \,\,\xymatrix@=15pt{
\dots  {u_{1}} \ar@/^0.9pc/[r]^{f_1} \ar@/_0.9pc/[r]_{e_1}  &  {u_{0}} \ar@/^0.9pc/[r]^{f_0} \ar@/_0.9pc/[r]_{e_0} & {u_{-1}}  \ar@/^0.9pc/[r]^{f_{-1}}  \ar@/_0.9pc/[r]_{e_{-1}} & \quad \cdots
}
\end{equation*}
\end{exm}

If $E$ is any graph, and $w : E^1 \to \G$ any function, we extend $w$ to $E^{*}$ by
defining $w(v) = 0$ for $v \in E^0$, and $w(\a_{1}\cdots \a_{n}) = w(\a_{1})\cdots
w(\a_{n})$. We obtain from \cite[Lemma 2.3]{kp} a continuous cocycle
$\widetilde{w}:\mg_{E}\xra \G$ satisfying
\[
    \widetilde{w}(\a x, |\a|-|\b|, \b x) = w(\a) w(\b)^{-1}.
\]
By Lemma \ref{ample} the skew-product groupoid $\mg_{E}\times \G$ is a $\G$-graded ample
groupoid. For each (possibly infinite) path $x = e^1 e^2 e^3 \dots \in E$ and each $\g
\in \G$ there is a path $x_\g$ of $\overline{E}$ given by
\begin{equation}\label{eq:pathlift}
x_\g = e^1_\g e^2_{w(e_1)^{-1}\g} e^3_{w(e^1e^2)^{-1}\g} \dots.
\end{equation}
There is an isomorphism
\[
f : \mg_{E}\times \G \longrightarrow \mg_{\overline{E}}
\]
of groupoids such that $f((x, k, y), \g) = (x_{\widetilde{w}(x,k,y)\g}, k, y_\g)$ (see~\cite[Theorem 2.4]{kp}).
The grading of the skew-product $\mg_{E}\times \G$ induces a grading of
$\mg_{\overline{E}}$, and the isomorphism $f$ respects the gradings of the two groupoids,
and so induces a graded isomorphism of Steinberg algebras
\begin{equation*}
\widetilde{f}:A_{R}(\mg_{E}\times \G) \longrightarrow A_{R}(\mg_{\overline{E}}).
\end{equation*}
Set $g=\widetilde{f}^{-1}:A_{R}(\mg_{\overline{E}}) \longrightarrow A_{R}(\mg_{E}\times
\G)$. Then
\begin{align}
\label{inducediso}
    g(1_{Z(v_{\g})}) &= 1_{Z(v)\times \{\g\}} \text{~for~} v\in E^{0} \text{~and~} \g\in\G, \nonumber\\
    g(1_{Z(e_\a, r(e)_{w(e)^{-1}\a})}) &= 1_{Z(e, r(e))\times \{w(e)^{-1}\a\}} \text{~for~} e\in E^{1} \text{~and~} \a\in\G, \\
    g(1_{Z(r(e)_{w(e)^{-1}\a}, e_\a)}) &= 1_{Z(r(e), e)\times \{\a\}} \text{~for~} e\in E^{1} \text{~and~} \a\in\G. \nonumber
\end{align}

Let $\phi: A_{R}(\mg_{E}\times \G) \to A_{R}(\mg_{E})\#\G$ be the isomorphism of
Theorem~\ref{isothm}, let $g : A_{R}(\mg_{\overline{E}}) \to A_{R}(\mg_{E}\times \G)$ be
the isomorphism~\eqref{inducediso}, let $\pi_E : L_R(E) \to A_R(\mg_E)$ and
$\pi_{\overline{E}} : L_R(\overline{E}) \to A_R(\mg_{\overline{E}})$ be as
in~\eqref{assteinberg}, and let $\widetilde{\pi}_{E} : L_R(E)\#\G \to A_R(\mg_E)\#\G$ be
given by $\widetilde{\pi}_E(xp_{\g}) = \pi_{E}(x)p_{\g}$ for $x\in L_{R}(E)$ and $\g\in
\G$. Define $\phi' := \widetilde{\pi}_{E}^{-1} \c \phi \c g \c \pi_{\overline{E}}$. Then
we have the following commuting diagram:
\begin{equation}\label{eq:comm diag}
\begin{split}
\xymatrix{ L_{R}(\overline{E})\ar[r]^{\phi'}\ar[d]^{\pi_{\overline{E}}}&L_{R}(E)\#\G\ar[d]^{\widetilde{\pi}_{E}}\\
A_{R}(\mg_{\overline{E}})\ar[r]^{\phi\circ g}&A_{R}(\mg_{E})\#\G.}
\end{split}
\end{equation}

\begin{cor} \label{coriso}
The map $\phi':L_{R}(\overline{E})\xra L_{R}(E)\#\G$ is an isomorphism of $\G$-graded
algebras such that $\phi'(v_{\b})=vp_{\b}$, $\phi'(e_{\a})=ep_{w(e)^{-1}\a}$ and
$\phi'(e^{*}_{\a})=e^{*}p_{\a}$ for $v\in E^{0}$, $e \in E^{1}$, and $\a,\b\in\G$.
\end{cor}
\begin{proof}
Since all the homomorphisms in the diagram~\eqref{eq:comm diag} preserve gradings of
algebras, the map $\phi':L_{R}(\overline{E})\longrightarrow L_{R}(E)\#\G$ is an
isomorphism of $\G$-graded algebras. For each vertex $v_{\g}\in \overline{E}^{0}$ and
each edge $e_{\a}\in \overline{E}^{1}$, we have
\begin{align*}
\phi'(v_{\g})
    &= (\widetilde{\pi}_{E}^{-1}\c\phi\c g)(1_{Z(v_{\g})})
     = (\widetilde{\pi}_{E}^{-1}\c\phi)(1_{Z(v)\times\{\g\}})
     = \widetilde{\pi}_{E}^{-1}(1_{Z(v)}p_{\g})=vp_{\g},\\
\phi'(e_\a)
    &= (\widetilde{\pi}_{E}^{-1}\c\phi\c g)(1_{Z(e_\a, r(e)_{w(e)^{-1}\a})})
     = (\widetilde{\pi}_{E}^{-1}\c\phi)(1_{Z(e, r(e))\times\{w(e)^{-1}\a\}})
     = \widetilde{\pi}_{E}^{-1}(1_{Z(e, r(e))} p_{w(e)^{-1}\a})
     = e p_{w(e)^{-1}\a},\\
\intertext{and}
\phi'(e^{*}_\a)
    &= (\widetilde{\pi}_{E}^{-1}\c\phi\c g)(1_{Z(r(e)_{w(e)^{-1}\a}, e_\a)})
     = (\widetilde{\pi}_{E}^{-1}\c\phi)(1_{Z(r(e),e)\times\{\a\}})
     = \widetilde{\pi}_{E}^{-1}(1_{Z(r(e),e)}p_{\a})
     = e^{*}p_{\a}.  \qedhere
\end{align*}
\end{proof}

In~\cite{kp}, Kumjian and Pask show that given a free action of  a group $\G$ on a graph
$E$, the crossed product $C^{*}(E)\times \G$ by the induced action is strongly Morita
equivalent to $C^{*}(E/\G)$, where $E/\G$ is the quotient graph and obtained an
isomorphism similar to Corollary~\ref{coriso}  for graph $C^*$-algebras.
Corollary~\ref{coriso} shows that this isomorphism already occurs on the algebraic level
(see~\S\ref{aidancalgebra}), so the following diagram commutes:
\begin{equation*}
\xymatrix{ L_{\mathbb C}(\overline E) \ar[d] \ar[r] &L_{\mathbb C}(E)\#\G \ar[d]\\
C^*(\overline E) \ar[r]& C^*(E) \times \G.
}
\end{equation*}

\begin{rmk}
In \cite{gr}, Green showed that the theory of coverings of graphs with relations and the
theory of graded algebras are essentially the same. For a $\G$-graded algebra $A$, Green
constructed a covering of the quiver of $A$ and showed that the category of
representations of the covering satisfying a certain set of relations is equivalent to
the category of finite dimensional graded $A$-modules.

For any graph $E$ and a function $w:E^{1}\xra \G$, we consider the smash product of a
quotient algebra of the path algebra of $E$ with the group $\G$. Let $K$ be a field, $E$
a graph and $w:E^{1}\xra \G$ a weight map. Denote by $KE$ the path algebra of $E$. A
\emph{relation in $E$} is a $K$-linear combination $\sum_{i}k_{i}q_{i}$ with $q_{i}$
paths in $E$ having the same source and range. Let $r$ be a set of relations in $E$ and
$\langle r \rangle$ the two sided ideal of $KE$ generated by $r$. Set
$$\mathcal{A}_{r}(E)=KE/\langle r\rangle.$$ We denote by $\overline{r}$ the lifting of
$r$ in $\overline{E}$. For each finite path $p=e^1e^2\cdots e^n$ in $E$ and $\g\in \G$, there is a path $p^{\g}$ of $\overline{E}$ given by $$p^{\g}=e^1_{\prod_{i=1}^n w(e^i)\g}\cdots e^{n-1}_{\prod_{i=n-1}^n w(e^i)\g}e^n_{w(e^n)\g},$$ similar as ~\eqref{eq:pathlift}. More precisely, for each relation
$\sum_{i}k_{i}q_{i}\in r$ and each $\g\in \G$, we have
\[\textstyle
\sum_{i}k_{i}q_{i}^{\g} \in\overline{r}.
\] 

Set
$$\mathcal{A}_{\overline{r}}(\overline{E})=K\overline{E}/\langle \overline{r}\rangle.$$

We prove that $\mathcal{A}_{\overline{r}}(\overline{E}) \cong \mathcal{A}_{r}(E)\#\G$.
Define $h : K\overline{E} \xra \mathcal{A}_{r}(E)\#\G$ by $h(v_{\g})=vp_{\g}$ and
$h(e_\a)=ep_{w(e)^{-1}\a}$ for $v\in E^{0}$, $e\in E^{1}$ and $\a,\g\in\G$. Since $h$
annihilates the relations $\overline{r}$, it induces a homomorphism
\[
\overline{h} : \mathcal{A}_{\overline{r}}(\overline{E})\longrightarrow\mathcal{A}_{r}(E)\#\G.
\]
We show that $\overline{h}$ is an isomorphism. For injectivity, suppose that
$x=\sum_{i=1}^{m}\lambda_{i}\xi_{i}\in \mathcal{A}_{\overline{r}}(\overline{E})$ with
$\lambda_{i}\in K$ and $\xi_{i}$ pairwise distinct paths in $\overline{E}$. Each $\xi_i$ has the form of  $(\xi'_i)^{\a_i}$ for some $\xi'_i\in E^*$ and $\a_i\in\G$. If
$\overline{h}(x)=0$, then $h(x)=\sum_{i=1}^{m}\lambda_{i}\xi_{i}'p_{\a_{i}}=0$. Suppose
that the $\a_i$ are not distinct; so by rearranging, we can assume that
$\a_{1}=\cdots=\a_{k}$ for some $k\leq m$. Then $\sum_{i=1}^{k}\lambda_{i}\xi_{i}'=0$ in
$\mathcal{A}_{r}(E)$. Observe that $\sum_{i=1}^{k}\lambda_{i}\xi_{i}'=0$ in
$\mathcal{A}_{r}(E)$ implies $\sum_{i=1}^{k}\lambda_{i}\xi_{i}=0$ in
$\mathcal{A}_{\overline{r}}(\overline{E})$. Hence $x=0$, implying $\overline{h}$ is
injective. For surjectivity, fix $\eta$ in $E^*$ and $\g\in\G$. Then $h(\eta^{\g}) =\eta
p_{\g}$ by definition. Since the elements $\{\eta p_\g \mid \eta \in E^*, \g \in \G\}$
span $\mathcal{A}_r(E)\#\G$, we deduce that $\overline{h}$ is surjective. Thus
$\overline{h}$ is an isomorphism as claimed.
\end{rmk}

\subsection{\texorpdfstring{The monoid $\pmb{\VV^{\gr}(L_{K}(E))}$}{The monoid of graded modules of a Leavitt path algebra}}\label{stlib}
In this subsection, we consider the Leavitt path algebra $L_{K}(E)$ over a field $K$.
Ara, Moreno and Pardo \cite{amp} showed that for a Leavitt path algebra associated to the
graph $E$, the monoid $\mathcal{V}(L_{K}(E))$ is entirely determined by elementary
graph-theoretic data. Specifically, for a row-finite graph $E$, we define $M_E$ to be the
abelian monoid presented by $E^{0}$ subject to
\begin{equation}
\label{monoidrelation}
v=\sum_{e\in s^{-1}(v)}r(e)
\end{equation}
for every $v\in E^{0}$ that is not a sink. Theorem~3.5 of~\cite{amp} says that
$\mathcal{V}(L_{K}(E)) \cong M_E$.

There is an explicit description \cite[\S 4]{amp} of the congruence on the free abelian
monoid given by the defining relations of $M_{E}$. Let $F$ be the free abelian monoid on
the set $E^{0}$. The nonzero elements of $F$ can be written in a unique form up to
permutation as $\sum_{i=1}^{n}v_{i}$, where $v_{i}\in E^{0}$. Define a binary relation
$\xra_{1}$ on $F\setminus\{0\}$ by $\sum_{i=1}^{n}v_{i}\xra_{1}\sum_{i\neq
j}v_{i}+\sum_{e\in s^{-1}(v_{j})}r(e)$ whenever $j\in \{1, \cdots, n\}$ is such that
$v_{j}$ is not a sink. Let $\xra$ be the transitive and reflexive closure of $\xra_{1}$
on $F\setminus\{0\}$ and $\sim$ the congruence on $F$ generated by the relation $\xra$.
Then $M_{E}=F/\sim$.

Ara and Goodearl defined analogous monoids $M(E, C, S)$ and constructed natural
isomorphisms $M(E, C, S)\cong \mathcal{V}(CL_{K}(E, C, S))$ for arbitrary separated
graphs (see \cite[Theorem~4.3]{ag}). The non-separated case reduces to that of ordinary
Leavitt path algebras, and extends the result of \cite{amp} to non-row-finite graphs.

Following \cite{ag,ag2}, we recall the definition of $M_{E}$ when $E$ is not necessarily
row-finite. In \cite[\S4.1]{ag2} the generators $v\in E^{0}$ of the abelian monoid
$M_{E}$ for $E$ are supplemented by generators $q_{Z}$ as $Z$ runs through all nonempty
finite subsets of $s^{-1}(v)$ for infinite emitters $v$. The relations are
\begin{itemize}
\item[(1)] $v=\sum_{e\in s^{-1}(v)}r(e)$ for all regular vertices $v\in E^{0}$;
\item[(2)] $v=\sum_{e\in Z}r(e)+q_{Z}$ for all infinite emitters $v\in E^{0}$ and all
\item[(3)] $q_{Z_{1}}=\sum_{e\in Z_{2}\setminus Z_{1}}r(e)+q_{Z_{2}}$ for all
    nonempty finite sets $Z_{1}\subseteq Z_{2}\subseteq s^{-1}(v)$, where $v\in
    E^{0}$ is an infinite emitter.
\end{itemize}

An abelian monoid $M$ is \emph{cancellative} if it satisfies full cancellation, namely,
$x + z = y + z$ implies $x = y$, for any $x, y, z \in M$. In order to prove that the
\emph{graded} monoid associated to any Leavitt path algebra is cancellative
(Corollary~\ref{corinj}), we will need to know that the monoid associated to Leavitt path
algebras of acyclic graphs are cancellative.

\begin{lem}\label{iffcancellative}
Let $E$ be an arbitrary graph. The monoid $M_{E}$ is cancellative if and only if no cycle
in $E$ has an exit. In particular, if $E$ is acyclic, then $M_E$ is cancellative.
\end{lem}
\begin{proof}
We first claim that $M_{E}$ is cancellative for any row-finite acyclic graph $E$. By
\cite[Lemma 3.1]{amp}, the row-finite graph $E$ is a direct limit of a directed system of
its finite complete subgraphs $\{E_{i}\}_{i\in X}$. In turn, the monoid $M_{E}$ is the
direct limit of $\{M_{E_{i}}\}_{i\in X}$ (\cite[Lemma 3.4]{amp}). We claim that $M_{E}$
is cancellative. Let $x+u=y+u$ in $M_{E}$, where $x, y,u$ are sum of vertices in $E$. By
\cite[Lemma 4.3]{amp}, there exists $b\in F$ (sum of vertices in $E$) such that $x+u\xra
b$ and $y+u\xra b$. Observe that vertices involved in this transformations are finite.
Thus there is a finite graph $E_i$ such that all these vertices are in $E_i$. It follows
that in $M_{E_i}$ we have $x+u\xra b$ and $y+u\xra b$. Thus $x+u=y+u$ in $M_{E_i}$. Since
the subgraph $E_{i}$ of $E$ is finite and acyclic, $M_{E_i}$ is a direct sum of copies of
$\mathbb N$ (as $L_K(E_i)$ is a semi-simple ring) and thus is cancellative.  So $x=y$ in
$M_{E_i}$ and so the same in $M_E$. Hence, $M_{E}$ is cancellative.

We now show that it suffices to consider the case where $E$ is a  row-finite graph in which no cycle has an exit. To see this, let $E$ be any graph, and let $E_d$ be its Drinen--Tomforde
desingularisation \cite{dt}, which is row-finite. Then $L_K(E)$ and $L_K(E_d)$ are Morita equivalent, and so
$M_E \cong M_{E_d}$  \cite[Theorem 5.6]{abramspino}. So $M_E$ has cancellation if
and only if $M_{E_d}$ has cancellation. Since no cycle in $E$ has an exit if and only if
$E_d$ has the same property, it therefore suffices to prove the result for row-finite graph $E$ in which no cycle has an exit.

Finally, we show that for any row-finite graph $E$ in which no cycle has an exit, the monoid $M_E$ is
cancellative. For this, fix a set $C \subseteq E^1$ such that $C$ contains exactly one
edge from every cycle in $E$ \cite{si3}. Let $F$ be the subgraph of $E$ obtained by
removing all the edges in $C$. We claim that $M_F \cong M_E$. To see this, observe that
they have the same generating set $F^0 = E^0$, and the generating relation
$\stackrel{F}{\xra}_1$ is contained in $\stackrel{E}{\xra}_1$. So it suffices to show
that $\stackrel{E}{\xra}_1 \subseteq \stackrel{F}{\xra}$. For this, note that for $v \in
E^0$, we have $s_E^{-1}(v) = s_F^{-1}(v)$ unless $v = s(e)$ for some $e \in C$, in which
case $s_E^{-1}(v) = \{e\}$ and $s_F^{-1}(v) = \emptyset$. So it suffices to show that for
$e \in C$, we have $s(e) \stackrel{F}{\xra} r(e)$. Let $p = e\a_2\a_3\dots\a_n$ be the
cycle in $E$ containing $e$. Then
\[
r(e) \stackrel{F}{\xra}_1 s(\a_2)
    \stackrel{F}{\xra}_1 r(\a_2)
    \stackrel{F}{\xra}_1 s(\a_3)
    \stackrel{F}{\xra}_1 r(\a_3)
    \stackrel{F}{\xra}_1 \cdots
    \stackrel{F}{\xra}_1 s(\a_n)
    \stackrel{F}{\xra}_1 r(\a_n)
    \stackrel{F}{\xra}_1 s(e).
\]
So $M_F \cong M_E$ as claimed. So the preceding paragraphs show that $M_E$ is
cancellative.

Now suppose that $E$ has a cycle with an exit; say $p = \a_1\dots\a_n$ has an exit $\a$.
Without loss of generality, $s(\a) = s(\a_n)$ and $\a \not= \a_n$. Write
\[
s(p) E^{\le n} = \{q \in E^* : s(q) = s(p),\text{ and } |q| = n\text{ or } |q|<n\text{ and } r(q)\text{ is not  regular}\}.
\]
Let $p' := \a_1\dots\a_{n-1}\a$ and $X := s(p) E^{\le n} \setminus \{p, p'\}$. A simple
induction shows that
\[
s(p) \xra \sum_{q \in s(\a)E^{\le n}} r(q)
    = r(p) + r(p') + \sum_{q \in X} r(q)
    = s(p) + r(p') + \sum_{q \in X} r(q).
\]
Since $r(p') \not= 0$ in $M_E$, it follows that $M_E$ does not have cancellation.
\end{proof}

In order to compute the monoid $\VV^{\gr}(L_{K}(E))$ for an arbitrary graph $E$, we
define an abelian monoid $M_{E}^{\gr}$ such that the generators $\{a_{v}({\g}) \mid  v\in
E^{0}, \g\in \G\}$ are supplemented  by generators $b_{Z}(\g)$ as $\g\in\G$ and $Z$ runs
through all nonempty finite subsets of $s^{-1}(u)$ for infinite emitters $u\in E^{0}$.
The relations are
\begin{enumerate}
\item[(1)] $a_{v}({\g})=\sum\limits_{e\in s^{-1}(v)}a_{r(e)}(w(e)^{-1}\g)$  for all
    regular vertices $v\in E^{0}$ and $\g\in\G$;
\item[(2)] $a_{u}({\g})=\sum\limits_{e\in Z}a_{r(e)}({w(e)^{-1}\g})+b_{Z}({\g})$ for
    all $\g\in\G$, infinite emitters $u\in E^{0}$ and nonempty finite subsets
    $Z\subseteq s^{-1}(u)$;
\item[(3)] $b_{Z_{1}}({\g})=\sum\limits_{e\in Z_{2}\setminus
    Z_{1}}a_{r(e)}({w(e)^{-1}\g})+b_{Z_{2}}({\g})$ for all $\g\in\G$, infinite
    emitters $u\in E^{0}$ and nonempty finite subsets $Z_{1}\subseteq Z_{2}\subseteq
    s^{-1}(u)$.
\end{enumerate}

The group $\Gamma$ acts on the monoid $M_{E}^{\gr}$ as follows. For any $\b\in \G$,
\begin{equation} \label{groupaction1}
    \b\cdot a_{v}({\g})=a_{v}({\b\g})\qquad\text{ and }\qquad \b\cdot b_{Z}({\g})=b_{Z}({\b\g}).
\end{equation}

There is a surjective monoid homomorphism $\pi \colon M_E^{\gr} \to M_E$ such that $\pi
(a_v(\gamma)) = v$ and $\pi ( b_Z(\gamma)) = q_Z$ for $v\in E^0$ and nonempty finite
subset $Z\subset s^{-1}(u)$, where $u$ is an infinite emitter. $\pi$ is
$\Gamma$-equivariant in the sense that $\pi (\beta \cdot x)= \pi (x)$ for all $\beta \in
\Gamma $ and $x\in M^{\gr}_E$.

Recall the covering graph $\overline{E}$ from  \S\ref{covergraph}. Let
$L_{K}(\overline{E})\text{-}{\rm Mod}$ be the category of unital left
$L_{K}(\overline{E})$-modules and $L_{K}(E)\-\Gr$ the category of graded unital left
$L_{K}(E)$-modules. The isomorphism $\phi':L_{K}(\overline{E})\xrightarrow[]{\sim}
L_{K}(E)\#\G$ of Corollary~\ref{coriso} and Proposition~\ref{proposition} yield an
isomorphism of categories
\begin{equation}\label{isocat}
\Phi : L_{K}(E)\text{-}{\Gr} \longrightarrow L_{K}(\overline{E})\text{-}{\rm Mod}.
\end{equation}

\begin{lem} \label{lemmaimage}
Let $E$ be an arbitrary graph, $\G$ a group and $w:E^{1}\xra \G$ a function.
\begin{enumerate}
\item[(1)] Fix a path $\eta$ in $E$, and $\b\in \G$, and let $\overline{\eta}=\eta_{\b^{-1}}$ be the path in
    $\overline{E}$ defined at~\eqref{eq:pathlift}. Then
    $\Phi\left((L_{K}(E)\eta\eta^{*})(\b)\right)\cong L_{K}(
    \overline{E})\overline{\eta}\overline{\eta}^{*}$. In particular,
    $\Phi((L_{K}(E)v)(\b))\cong L_{K}(\overline{E})v_{\b^{-1}}$.
\item[(2)] Let $u \in E^0$ be an infinite emitter, and let $Z\subseteq s_{E}^{-1}(u)$
    be a nonempty finite set. Fix $\b \in \G$, and let $\overline{Z} = \{e_{\b^{-1}}
    \mid e \in Z\}$. Then $u_{\b^{-1}}$ is an infinite emitter in $\overline{E}$ and
    $\overline{Z}$ is a nonempty finite subset of
    $s^{-1}_{\overline{E}}(u_{\b^{-1}})$. Moreover, $\Phi(L_{K}(E)(u-\sum_{e\in
    Z}ee^{*})(\b))\cong
    L_{K}(\overline{E})(u_{\b^{-1}}-\sum_{f\in\overline{Z}}ff^{*})$.
\end{enumerate}
\end{lem}
\begin{proof}
We prove (1). By the isomorphism of algebras in Corollary \ref{coriso}, we have
\begin{equation*}
\begin{split}
L_{K}(\overline{E})\overline{\eta}\overline{\eta}^{*} &\cong (L_{K}(E)\#\G)\eta\eta^{*}
p_{\b^{-1}}.
\end{split}
\end{equation*}
We claim that $f: \Phi((L_{K}(E)\eta\eta^{*})(\b))\longrightarrow
(L_{K}(E)\#\G)\eta\eta^{*} p_{\b^{-1}}$ given by $f(y)=yp_{\b^{-1}}$ is an isomorphism of
left $L_{K}(\overline{E})$-modules. It is clearly a group isomorphism. To see that it is
an $L_{K}(\overline{E})$-module morphism, note that $(rp_{\g})y=ry_{\g}$ for $y\in
(L_{K}(E)\eta\eta^{*})(\b)$ and $y_{\g}$ a homogeneous element of degree $\g$. We have
$y\in L_{K}(E)_{\g\b}\eta\eta^{*}$, yielding $f((rp_{\g})y)=ry_{\g}p_{\b-1} =
(rp_{\g})(yp_{\b^{-1}})=rp_{\g}f(y)$. The proof for (2) is similar.
\end{proof}

Recall from \S\ref{smashp} that there is a shift functor $\widetilde{\mathcal{S}}_{\a}$
on $L_{K}(E)\#\G\-\Mod$ for each $\a\in\G$. So the isomorphism
$\phi':L_{K}(\overline{E})\xrightarrow[]{\sim}L_{K}(E)\#\G$ of Corollary \ref{coriso}
yields a shift functor $\TT_\a$ on $L_{K}(\overline{E})\-\Mod$. This in turn induces a
homomorphism $\mathcal{T}_{\a}:\VV(L_{K}(\overline{E}))\xra \VV(L_{K}(\overline{E}))$,
giving a $\G$-action on the monoid $\VV(L_{K}(\overline{E}))$.

Fix $v_{\g}\in \overline{E}^{0}$, an infinite emitter $u_\b \in \overline{E}^0$, and a
finite $Z \subseteq s^{-1}_{\overline{E}}(u_\b)$. Write $Z\cdot\a^{-1} = \{e_{\b\a^{-1}}
\mid e_\b \in Z\}$. We claim that
\begin{equation}
\label{shift}
\mathcal{T}_{\a}([L_{K}(\overline{E})v_{\g}])=[L_{K}(\overline{E})v_{\g\a^{-1}}]
    \qquad\text{ and }\qquad
\mathcal{T}_{\a}([L_{K}(\overline{E})(u_{\b}-\sum_{e\in Z}ee^{*})])=[L_{K}(\overline{E})(u_{\b\a^{-1}}-\sum_{f\in Z\cdot a^{-1}}ff^{*})].
\end{equation}

To see the first equality in~\eqref{shift}, we use Lemma~\ref{lemmaimage} to see that
\begin{equation*}
\Phi( L_{K}(E)v(\g^{-1}))=L_{K}(\overline{E})v_{\g}
    \qquad\text{ and }\qquad
\Phi( L_{K}(E)v(\a\g^{-1}))=L_{K}(\overline{E})v_{\g\a^{-1}}.
\end{equation*}
Using the commutative diagram~\eqref{inducedshift} at the second equality, we see that
\begin{equation*}
\TT_{\a}(L_{K}(\overline{E})v_{\g})
    =(\TT_{\a}\c \Phi)( L_{K}(E)v(\g^{-1}))
    =(\Phi\c\TT_{\a})( L_{K}(E)v(\g^{-1}))
    =\Phi( L_{K}(E)v(\a\g^{-1}))
    =L_{K}(\overline{E})v_{\g\a^{-1}}.
\end{equation*}
The proof for the second equality in \eqref{shift} is similar.

The group $\G$ acts on the monoid $M_{\overline{E}}$  as follows. Again fix $v_{\g}\in
\overline{E}^{0}$, an infinite emitter $u_\b \in \overline{E}^0$, and a finite $Z
\subseteq s^{-1}_{\overline{E}}(u_\b)$, and write $Z\cdot\a^{-1} = \{e_{\b\a^{-1}} \mid
e_\b \in Z\}$. Then
\begin{equation} \label{groupaction2}
\a\cdot v_{\g}=v_{\g\a^{-1}}
    \qquad\text{ and }\qquad
\a\cdot q_{Z}=q_{Z\cdot\a^{-1}}.
\end{equation}

\begin{prop} \label{propformonoid}
Let $E$ be an arbitrary graph, $K$ a field, $\G$ a group and $w:E^{1}\xra \G$ a function.
Let $\mathcal{A}=L_{K}(E)$ and $\mathcal{\overline{A}}=L_{K}(\overline{E})$. Then  the
monoid $\mathcal{V}^{\gr}(\mathcal{A})$ is generated by $[\mathcal{A}v(\a)]$ and
$[\AA(u-\sum_{e\in Z}ee^{*})(\b)] $, where $v\in E^{0}, \a,\b\in \G$ and $Z$ runs through
all nonempty finite subsets of $s^{-1}(u)$ for infinite emitters $u\in E^{0}$. Given an
infinite emitter $u \in E^0$, a finite nonempty set $Z \subseteq s^{-1}(u)$, and $\b\in
\G$, write $Z_{\b^{-1}} := \{e_{\b^{-1}} : e \in Z\} \subseteq s^{-1}_{\overline{E}}(u_{\b^{-1}})$. Then there
are $\G$-module isomorphisms
\begin{equation}
\label{isomonoid}
\mathcal{V}^{\gr}(\mathcal{A}) \cong \mathcal{V}(\mathcal{\overline{A}}) \cong M_{\overline{E}} \cong M_{E}^{\gr},
\end{equation}
that satisfy
\begin{equation*}
[\mathcal{A}v(\a)] \mapsto [\mathcal{\overline{A}}v_{\a^{-1}}]\mapsto[v_{\a^{-1}}]\mapsto[a_{v}({\a})]
\end{equation*}
for all $v \in E^0$ and $\a \in \G$, and
\begin{equation*}
[\AA(u-\sum_{e\in Z}ee^{*})(\b)]
    \mapsto [\mathcal{\overline{A}}(u_{\b^{-1}} - \sum_{\overline{e}\in Z_{\b^{-1}}}\overline{e}\overline{e}^{*})]\mapsto[q_{Z_{\b^{-1}}}]\mapsto[b_{Z}({\b})]
\end{equation*}
for every infinite emitter $u$, finite nonempty $Z \subseteq s^{-1}(u)$, and $\b \in \G$.
\end{prop}
\begin{proof}
Let $P$ be a graded finitely generated projective left $\mathcal{A}$-module. We claim
that the isomorphism $\Phi:\mathcal{A}\text{-}{\Gr}\xra
\mathcal{\overline{A}}\text{-}{\rm Mod}$ in \eqref{isocat} preserves the finitely
generated projective objects. Since $\Phi$ is an isomorphism of categories, $\Phi(P)$ is
projective. Observe that $P$ has finite number of homogeneous generators $x_{1}, \cdots,
x_{n}$ of degree $\g_{i}$. By the $\overline{\AA}$-action of $\Phi(P)$, we have the
following equalities:
\begin{enumerate}
\item[(1)] if $v\in E^{0}, \g\in \G$, then
\begin{equation} \label{fg1}
    v_{\g}x_{i}=vp_{\g}x_{i}=
\begin{cases}vx_{i}, & \text{if~} \g_{i}=\g;\\
0, &\text{otherwise;}
\end{cases}
\end{equation}
\item[(2)] if $e:u\xra v\in E^{1}, w(e)=\b$ and $\g\in \G$, then
\begin{equation}
    \label{fg2}
 e_{\g}x_{i}=ep_{\b^{-1}\g}x_{i}=
\begin{cases}ex_{i}, & \text{if~} \g_{i}=\b^{-1}\g;\\
0, &\text{otherwise; and}
\end{cases}
\end{equation}
\item[(3)] if $ e:u\xra v\in E^{1}, w(e)=\b$ and $\g\in \G$, then
\begin{equation}
    \label{fg3}
 e^{*}_{\g}x_{i}=e^{*}p_{\g}x_{i}=\begin{cases}e^{*}x_{i}, & \text{if~} \g_{i}=\g;\\
0, &\text{otherwise.}
\end{cases}
\end{equation}
\end{enumerate}
So for $y\in \Phi(P)$, we can express $y=\sum_{i=1}^{n}r_{i}x_{i}$ for some $r_{i}\in
\AA$. Fix $i \le n$ and paths $\eta, \tau$ in $E$ satisfying $r(\eta)=r(\tau)$.
Then~\eqref{fg1},~\eqref{fg2}, and~\eqref{fg3} give
\begin{equation} \label{equality}
\tau\eta^{*}x_{i}=\tau_{w(\tau)w(\eta)^{-1}\g_i} (\eta_{\g_i})^{*}x_{i}.
\end{equation}
Since $y=\sum_{i=1}^{n}r_{i}x_{i}=\sum_{i=1}^{n}\sum_{h\in \G}r_{i, h}x_{i}$ with $r_{i,
h}$ a homogeneous element of degree $h$, equation~\eqref{equality} gives $y\in
\overline{\AA}(\Phi(P))$. Thus $\Phi(P)$ is a finitely generated projective
$\mathcal{\overline{A}}$-module.

By \eqref{monoididempotent} and \eqref{grlocal}, there exists a homomorphism
$\VV^{\gr}(\AA)\xra \VV(\overline{\AA})$ sending $[P]$ to $[\Phi(P)]$ for a graded
finitely generated projective left $\AA$-module $P$. Applying \cite[Theorem 4.3]{ag} for
the non-separated case, we obtain the second monoid isomorphism
$\mathcal{V}(\mathcal{\overline{A}})\xrightarrow[]{\sim} M_{\overline{E}}$ in
\eqref{isomonoid}. Then for each graded finitely generated projective left $\AA$-module
$P$, the module $\Phi(P)$ in $\mathcal{\overline{A}}\text{-}{\rm Mod}$ is generated by
the elements $\mathcal{\overline{A}}v_{\a}$ and $\mathcal{\overline{A}}(u_{\b}-\sum_{e\in
Z'}ee^{*})$ that it contains. Combining this with Lemma~\ref{lemmaimage} gives the first
isomorphism of monoids. The last monoid isomorphism $M_{\overline{E}}\cong M_{E}^{\gr}$
follows directly by their definitions. By~\eqref{groupaction1},~\eqref{shift}
and~\eqref{groupaction2}, the monoid isomorphisms in~\eqref{isomonoid} are $\G$-module
isomorphisms.
\end{proof}

Recall the following classification conjecture~\cite{abramss,ap2014,haz2013}. Let $E$ and
$F$ be finite graphs. Then there is an order preserving $\mathbb Z[x,x^{-1}]$-module
isomorphism $\phi: K_0^{\gr}(L_K(E)) \rightarrow K_0^{\gr}(L_K(F))$ if and only if
$L_K(E)$ is graded Morita equivalent to $L_K(F)$. Furthermore, if
$\phi([L_K(E)]=[L_K(F)]$ then $L_K(E)\cong_{\gr} L_K(F).$

Note that $K_0(L_K(E))$ and $K^{\gr}_0(L_K(E))$ are the group completions of
$\VV(L_K(E))$ and $\VV^{\gr}(L_K(E))$, respectively.  Let  $\G=\Z$ and let $w:E^1
\rightarrow \Z$ be the function assigning $1$ to each edge. Then
Proposition~\ref{propformonoid} implies that there is an order preserving $\mathbb
Z[x,x^{-1}]$-module  isomorphism $K_{0}^{\gr}(L_{K}(E))\cong K_{0}(L_{K}(\overline{E}))$,
thus relating the study of a Leavitt path algebra over an arbitrary graph to the case of
acyclic graphs (see Example~\ref{onetwo3}).

The following corollary is the first evidence that $K^{\gr}_0(L_K(E))$ preserves all the
information of the graded monoid.

\begin{cor}\label{corinj}
Let $E$ be an arbitrary graph. Consider $L_K(E)$ as a graded ring with the grading
determined by the function $w : E^1 \to \Z$ such that $w(e) = 1$ for all $e$. Then
$\VV^{\gr}(L_{K}(E))$ is cancellative.
\end{cor}
\begin{proof}
By Proposition \ref{propformonoid}, we have $\VV^{\rm gr}(L_{K}(E))\cong
M_{\overline{E}}$. Since $\overline{E}=E\times \Z$ is an acyclic graph, the monoid
$M_{\overline{E}}$ is cancellative by Lemma~\ref{iffcancellative}. Hence $\VV^{\rm
gr}(L_{K}(E))$ is cancellative.
\end{proof}

For the next result we need to recall the notion of order-ideals of a monoid. An
\emph{order-ideal} of a monoid $M$ is a submonoid $I$ of $M$ such that $x + y \in I$
implies $x,y \in I$. Equivalently, an order-ideal is a submonoid $I$ of $M$ that is
hereditary in the sense that $x \le y$ and $y \in I$ implies $x \in I$. The set
$\mathcal{L}(M)$ of order-ideals of $M$ forms a (complete) lattice (see \cite[\S
5]{amp}). Given a subgroup $I$ of $K_0^{\gr}(A)$, we write $I^{+}=I\cap
K_0^{\gr}(A)^{+}$. We say that $I$ is a {\it graded ordered ideal} if $I$ is closed under
the action of $\mathbb Z[x,x^{-1}]$, $I=I^{+}-I^{+}$, and $I^{+}$ is an order-ideal.

Let $E$ be a graph. Recall that a subset $H \subseteq E^0$ is said to be hereditary if
for any $e \in E^1$ we have that $s(e)\in H$ implies $r(e)\in H$. A hereditary subset $H
\subseteq E^0$ is called saturated if whenever $0 < |s^{-1}(v)| < \infty$, then $\{r(e):
e\in E^1 \text{~and~} s(e)=v\}\subseteq H$ implies $v\in H$. If $H$ is a hereditary
subset, a breaking vertex of $H$ is a vertex $v \in E^0 \setminus H$ such that
$|s^{-1}(v)| = \infty$ but $0 < |s^{-1}(v) \setminus r^{-1}(H)| < \infty$. We write $B_H
:= \{v \in E^0\setminus H \mid v \text{ is a breaking vertex of } H\}$. We call $(H,S)$
an \emph{admissible pair} in $E^{0}$ if $H$ is a saturated hereditary subset of $E^0$ and
$S\subseteq B_H$.

Let $E$ be a row-finite graph. Isomorphisms between the lattice of saturated hereditary
subsets of $E^{0}$, the lattice $\mathcal{L}(M_{E})$, and the lattice of graded ideals of
$L_K(E)$ were established in \cite[Theorem 5.3]{amp}. Tomforde used the admissible pairs
$(H, S)$ of vertices to parameterise the graded ideals of $L_{K}(E)$ for a graph $E$
which is not row-finite (see \cite[Theorem 5.7]{tomforde}). In analogy, Ara and Goodearl
\cite{ag} proved that the lattice of those ideals of Cohn-Leavitt algebras $CL_{K}(E, C,
S)$ generated by idempotents is isomorphic to a certain lattice $\mathcal{A}_{C, S}$ of
admissible pairs $(H, G)$, where $H\subseteq E^0$ and $G\subseteq C$ (see
\cite[Definition 6.5]{ag} for the precise definition). There is also a lattice
isomorphism between $\mathcal{A}_{C, S}$ and the lattice $\mathcal{L}(M(E, C, S))$ of
order-ideals of $M(E, C, S)$. Specialising to the non-separated graph $E$, there is a
lattice isomorphism
\begin{equation}\label{latticeiso}
\mathcal{H} \cong \mathcal{L}(M_{E})
\end{equation}
between the lattice $\mathcal{H}$ of admissible pairs $(H, S)$ of $E^0$ and the lattice
$\mathcal{L}(M_{E})$ of order-ideals of the monoid $M_{E}$.

Let $E$ be a finite graph with no sinks. There is a one-to-one correspondence
\cite[Theorem 12]{roozbehhazrat2013} between the set of hereditary and saturated subsets
of $E^0$ and the set of graded ordered ideals of $K^{\rm gr}_0(L_K(E))$. The main theorem
of this section describes a one-to-one correspondence between the set of admissible pairs
$(H, S)$ of vertices and the set of graded ordered ideals of $K_0^{\gr}(L_{K}(E))$ for an
arbitrary graph $E$. To prove it, we first need to extend \cite[Lemma 4.3]{amp} to
arbitrary graphs. This may also be useful in other situations.

\begin{lem}\label{lem:confluence-general}
Let $E$ be an arbitrary graph and denote by $F$ the free abelian group generated by
$E^0\cup \{q_Z\}$, where $Z$ ranges over all the nonempty finite subsets of $s^{-1}(v)$
for infinite emitters $v$. Let $\sim$ be the congruence on $F$ such that $F/{\sim} =
M_E$. Let $\to_1$ be the relation on $F$ defined by $v+\alpha \to_1 \sum _{e\in
s^{-1}(v)} r(e) + \alpha$ if $v$ is a regular vertex in $E$, $v+\alpha \to_1 r(z) + q_{\{
z\}} + \alpha $ if $v\in E^0$ is an infinite emitter and $z\in s^{-1}(v)$, and also
$q_{Z}+\alpha  \to_1 r(z) + q_{Z\cup \{z\}}+ \alpha $, if $Z$ is a non-empty finite
subset of $s^{-1}(v)$ for an infinite emitter $v$ and $z\in s^{-1}(v)\setminus Z$. Let
$\to $ be the transitive and reflexive closure of $\to _1$. Then $\alpha \sim \beta $ in
$F$ if and only if there is $\gamma \in F$ such that $\alpha \to \gamma $ and $\beta \to
\gamma $.
\end{lem}
\begin{proof}
As in \cite[Alternative proof of Theorem 4.1]{ag2}, we write $M_E= \lim M(E',C', T')$,
where $E'$ ranges over all the finite complete subgraphs of $E$ and
$$C' = \{ s_{E'}^{-1}(v) \mid v\in (E')^0, \, \,  |s^{-1}_{E'}(v) | >0 \} ,\qquad
T' = \{ s_{E'}^{-1}(v)\in C' \mid v\in (E')^0, \, \, 0 < |s^{-1}_{E} (v) | <\infty \}.$$
Applying \cite[Construction 5.3]{ag}, we get that $M(E',C',T') = M_{\widetilde{E}}$ for
some finite graph $\widetilde{E}$. The vertices of $\widetilde{E}$ are the vertices of
$E$ and the elements of the form $q_Z$, where $Z\in C' \setminus T'$, and there is a new
edge $e_Z : v\to q_Z$ if the source of $Z$ is $v$.  If $\alpha \sim \beta $ in $F$, then
$[\alpha ]= [\beta]$ in $M_E$, and so there is $(E',C',T')$ as above such that $[\alpha]
= [\beta ]$ in $M(E',C',T')$. But since $M(E',C',T')= M_{\widetilde{E}}$, and
$\widetilde{E}$ is finite, we conclude from \cite[Lemma 4.3]{amp} that there is an
element $\gamma $ in the free monoid on $(E')^0 \cup \{q_Z\mid Z\in C'\setminus T' \}$
such that $\alpha \to \gamma $ and $\beta \to \gamma $. This implies that $\alpha \to
\gamma$ and $\beta \to \gamma $ in $F$.
\end{proof}

\begin{lem}
\label{lem:lattices-of-ideals} Let $E$ be an arbitrary graph and $K$ a field. Consider
$L_K(E)$ as a graded ring with the grading determined by the function $w : E^1 \to \Z$
such that $w(e) = 1$ for all $e$. Let $\mathcal{L}^{c}(M_{E}^{\rm gr})$ be the set of
order-ideals of $M^{\rm gr}_{E}$ which are closed under the $\mathbb{Z}$-action. Let $\pi
\colon M_E^{\gr} \to M_E$ be the canonical surjective homomorphism. Then the map $\phi
\colon \mathcal L (M_E) \to \mathcal L^c(M^{\gr}_E)$ defined by $\phi (I) = \pi^{-1}(I)$
is a lattice isomorphism.
\end{lem}
\begin{proof}
It is easy to show that the map $\phi$ is well-defined. The key to show the result is to prove the equality $\pi^{-1} (\pi (J))= J$ for any $J\in L^c(M^{\gr}_E)$.
The inclusion $J\subseteq \pi^{-1}(\pi (J))$ is obvious. To show the reverse inclusion $\pi^{-1}(\pi (J))\subseteq J$, denote by $F$ the free abelian group
on $E^0\cup \{ q_Z \}$, where $Z$ ranges over all the nonempty finite subsets of $s^{-1}(v)$ for infinite emitters $v$.
 Take $z\in \pi^{-1} (\pi (J))$. Then there is $y\in J$ such that $\pi (z) = \pi (y)$.
Now write
$$ z= \sum _i a_{v_i}(\gamma _i) + \sum _j b_{Z_j} (\lambda _j),\qquad y = \sum _i a_{v'_i}(\gamma _i')+ \sum _j b_{Z'_j}(\lambda_j') .$$
Then we have $\sum _i v_i + \sum _j q_{Z_j} = \pi (z) = \pi (y) = \sum _i v_i'+ \sum _j q_{Z_j'}$. By Lemma \ref{lem:confluence-general}, there is
$x= \sum _i w_i + \sum_j q_{W_j}$ such that $\pi (z)\to x$ and $\pi (y)\to x$ in $F$. Now using the same changes than in the paths $\pi (y)\to x$ and $\pi (z)\to x$, but lifted to $M^{\gr}_E$,
we obtain that  $y = \sum _i a_{w_i}(\eta_i) + \sum _j b_{W_j}(\nu_j)$ in $M^{\gr}_E$ and $z = \sum _i a_{w_i}(\eta'_i)+ \sum _j b_{W_j}(\nu'_j)$ in $M^{\gr}_E$. But now $y\in J$ and $J$ is an order ideal
of $M^{\gr}_E$, so it follows that $a_{w_i}(\eta_i) \in J$ for all $i$ and
$b_{W_j}(\nu_j) \in J$ for all $j$. Using that $J$ is invariant, we obtain
$a_{w_i}(\eta_i') \in J$ for all $i$ and $b_{W_j}(\nu_j') \in J$ for all $j$. Thus $z =
\sum _i a_{w_i}(\eta_i') + \sum _j b_{W_j}(\nu_j')\in J$ and we conclude the proof.

Now using that $J= \pi^{-1} (\pi (J))$, we can easily show that $\pi (J)$ is an order-ideal of $M_E$ and that the map $\phi$ is bijective, with $\phi^{-1}(J) =\pi (J)$.
\end{proof}

We can now state the main theorem of this section, which indicates that the graded
$K_0$-group captures the lattice structure of graded ideals of a Leavitt path algebra.

\begin{thm}\label{statelib2}
Let $E$ be an arbitrary graph and $K$ a field. Consider $L_K(E)$ as a graded ring with
the grading determined by the function $w : E^1 \to \Z$ such that $w(e) = 1$ for all $e$.
Then there is a one-to-one correspondence between the admissible pairs of $E^0$ and the
graded ordered ideals of $K_0^{\gr}(L_K(E))$.
\end{thm}
\begin{proof}
Let $\mathcal{H}$ be the set of all admissible pairs of $E^0$ and $\mathcal{L}(K^{\rm
gr}_0(\AA))$ the set of all graded ordered ideals of $K^{\rm gr}_0(\AA)$, where $\AA =
L_K (E)$. We first claim that there is a one-to-one correspondence between the
order-ideals of $M_{E}$ and order-ideals of $M^{\rm gr}_{E}$ which are closed under the
$\mathbb{Z}$-action. Let $\mathcal{L}^{c}(M_{E}^{\rm gr})$ be the set of order-ideals of
$M^{\rm gr}_{E}$ which are closed under the $\mathbb{Z}$-action.

The map $\phi:\mathcal{L}(M_{E}) \xra \mathcal{L}^{c}(M^{\rm gr}_E)$ has been defined in
Lemma \ref{lem:lattices-of-ideals}, where it is proved that it is a lattice isomorphism.

By Corollary \ref{corinj}, we have an injective homomorphism $\VV^{\rm gr}(\AA)\xra
K^{\rm gr}_0(\AA)$. By Proposition \ref{propformonoid}, there is a one-to-one
corespondence between the order-ideals of $M_{E}^{\rm gr}$ which are closed under the
$\mathbb{Z}$-action and the graded ordered ideals of $K^{\rm gr}_0(\AA)$. Finally
by~\eqref{latticeiso}, we have lattice isomorphisms
\begin{equation*}
\mathcal{H} \cong \mathcal{L}(M_{E})
    \cong \mathcal{L}^{c}(M^{\rm gr}_E)
    \cong \mathcal{L}(K^{\rm gr}_0(\AA)).\qedhere
\end{equation*}
\end{proof}

\section{Application: Kumjian--Pask algebras}
\label{sectionsix}

In this section we will use our result on smash products (Theorem~\ref{isothm}) to study
the structure of Kumjian--Pask algebras~\cite{pchr} and their graded $K$-groups. We will
see that the graded $K_0$-group remains a useful invariant for studying Kumjian--Pask
algebras. We deal exclusively with row-finite $k$-graphs with no sources: our analysis
for arbitrary graphs relied on constructions like desingularisation that are not
available in general for $k$-graphs. We briefly recall the definition of Kumjian--Pask
algebras and establish our notation. We follow the conventions used in the literature of
this topic (in particular the paths are written from right to left).

Recall that a \emph{graph of rank} $k$ or $k$-graph is a countable category $\Lambda =
(\Lambda^0,\Lambda,r,s)$ together with a functor $d: \Lambda \rightarrow \mathbb N^k$,
called the \emph{degree map}, satisfying the following factorisation property: if
$\lambda \in \Lambda$ and $d(\lambda) = m + n$ for some $m,n \in \mathbb N^k$, then there
are unique $\mu,\nu \in  \Lambda$ such that $d(\mu) = m, d(\nu) = n$, and $\lambda = \mu
\nu$. We say that $\Lambda$ is \emph{row finite} if $r^{-1}(v) \cap d^{-1}(n)$,
abbreviated $v\Lambda^n$ is finite for all $v \in \Lambda^0$ and $n \in \mathbb{N}^k$; we
say that $\Lambda$ has \emph{no sources} if each $v\Lambda^n$ is nonempty.

An important example is the $k$-graph $\Omega_k$ defined as a set by $\Omega_k = \{(m,n)
\in \mathbb{N}^k \times \mathbb{N}^k : m \le n\}$ with $d(m,n) = n-m$, $\Omega_k^0 =
\mathbb{N}^k$, $r(m,n) = m$, $s(m,n) = n$ and $(m,n)(n,p) = (m,p)$.

\begin{defi}
Let $\Lambda$ be a row-finite $k$-graph without sources and $K$ a field. The
\emph{Kumjian--Pask} $K$-algebra of $\Lambda$ is the $K$-algebra $\KP_K(\Lambda)$
generated by $\Lambda \cup \Lambda^*$ subject to the relations
\begin{enumerate}
\item[(KP1)] $\{v\in \Lambda^0\}$ is a family of mutually orthogonal idempotents
    satisfying $v = v^*$,
\item[(KP2)] for all $\lambda, \mu\in\Lambda$ with $r(\mu) = s(\lambda)$, we have
\[ \lambda \mu  = \lambda \circ \mu, \; \mu^* \lambda^* = (\lambda \circ \mu)^*, \;
 r(\lambda)\lambda = \lambda = \lambda s(\lambda), \;
  s(\lambda) \lambda^* = \lambda^* = \lambda^* r(\lambda),
\]
\item[(KP3)] for all $\lambda, \mu \in\Lambda$ with $d(\lambda) = d(\mu)$, we have
\[
\lambda^* \mu = \delta_{\lambda,\mu} s(\lambda),
\]
\item[(KP4)] for all $v\in\Lambda^0$ and all $n\in {\mathbb N}^k\setminus \{0\}$, we
    have
\[
v = \sum_{\lambda\in v\Lambda^n} \lambda\lambda^*.
\]
\end{enumerate}
\end{defi}

Let $\Lambda$ be a  a row-finite $k$-graph without sources and $\KP_K(\Lambda)$ the
Kumjian--Pask algebra of $\Lambda$. Following \cite[\S 2]{kp2000}, an infinite path in
$\Lambda$ is a degree-preserving functor $x : \Omega_{k}\xra \Lambda$. Denote the set of
all infinite paths by $\Lambda^{\infty}$. We define the relation of \emph{tail
equivalence} on the space of infinite path $\Lambda^{\infty}$ as follows: for $x,y \in
\Lambda^{\infty}$, we say $x$ is tail equivalent to $y$, denoted, $x\sim y$, if
$x(n,\infty)=y(m,\infty)$, for some $n,m \in \mathbb{N}^k$. This is an equivalence
relation. For $x\in \Lambda^\infty$, we denote by $[x]$ the equivalence class of $x$,
i.e., the set of all infinite paths which are tail equivalent to $x$. An infinite path
$x$ is called \emph{aperiodic} if $x(n,\infty)=x(m,\infty)$, $n,m \in \mathbb N^k$,
implies $n=m$.

We can form the skew-product $k$-graph, or covering graph, $\overline{\Lambda} = \Lambda
\times_d \mathbb{Z}^k$ which is equal as a set to $\Lambda \times \mathbb{Z}^k$, has
degree map given by $\overline{d}(\lambda, n) = d(\lambda)$, range and source maps
$r(\lambda,n) = (r(\lambda), n)$ and $s(\lambda, n) = (s(\lambda), n + d(\lambda))$ and
composition given by $(\lambda, n)(\mu, n + d(\lambda)) = (\lambda\mu, n)$.

As in the theory of Leavitt path algebras, one can model Kumjian--Pask algebras as
Steinberg algebras via the infinite-path groupoid of the $k$-graph (see
\cite[Proposition~5.4]{cp}). For the $k$-graph $\Lambda$,
\[
\mg_{\Lambda}=\big \{(x, l-m, y)\in \Lambda^{\infty}\times \mathbb{Z}^{k}\times \Lambda^{\infty} \mid x(l, \infty) = y(m, \infty)\big\}.
\]
Define range and source maps $r, s: \mg_{\Lambda}\xra \Lambda^{\infty}$ by $r(x, n, y)=x$
and $s(x, n, y)=y$. For $(x, n, y), (y, l, z)\in \mg_{\Lambda}$, the multiplication and
inverse are given by $(x, n, y)(y, l, z)=(x, n+l, z)$ and $(x, n, y)^{-1}=(y, -n, x)$.
$\mg_{\Lambda}$ is a groupoid with $\Lambda^{\infty}=\mg_{\Lambda}^{(0)}$ under the
identification $x\mapsto (x, 0, x)$. For $\mu,\nu \in \Lambda$ with $s(\mu)= s(\nu)$, let
$Z(\mu,\nu) := \{(\mu x, d(\mu)-d(\nu), \nu x) : x \in \Lambda^\infty, x(0) = s(\mu)\}$.
Then the sets $Z(\mu,\nu)$ comprise a basis of compact open sets for an ample Hausdorff
topology on $\mg_\Lambda$. There is a continuous $1$-cocycle $c : \mg_\Lambda \to
\mathbb{Z}^k$ given by $c(x,m,y) = m$.

For the skew-product $k$-graph $\overline{\Lambda}=\Lambda \times_d \mathbb{Z}^k$, we
have $\mg_{\overline{\Lambda}}\cong \mg_{\Lambda}\times_c \mathbb{Z}^k$ (see
\cite[Theorem~5.2]{kp2000}). Thus specialising Theorem~\ref{isothm} to this setting, we
have
\begin{equation}\label{ttteeett}
\KP_K(\overline{\Lambda}) \cong \KP_K(\Lambda)\#\mathbb{Z}^k.
\end{equation}
We will show that $\KP_K(\overline{\Lambda})$ is an ultramatricial algebra.

\begin{lem}\label{lem:summands}
For $n \in \mathbb{Z}^k$ define $B_n \subseteq \KP_K(\overline{\Lambda})$ by
\[
B_n =
\operatorname{span}_K\big\{(\lambda, n - d(\lambda))(\mu, n - d(\mu))^* \mid  \lambda,\mu \in
\Lambda, s(\lambda) = s(\mu)\big\}.
\]
Then $B_n$ is a subalgebra of $\KP_K(\overline{\Lambda})$ and there is an isomorphism
$B_n \cong \bigoplus_{v \in \Lambda^0} M_{\Lambda v}(K)$ that carries $(\lambda,
n-d(\lambda))(\mu, n - d(\mu))^*$ to the matrix unit $\mathbf{e}_{\lambda, \mu}$.
\end{lem}
\begin{proof}
For the first statement we just have to show that for any $\lambda,\mu,\eta,\zeta\in
\Lambda$ we have \[(\lambda, n-d(\lambda))(\mu, n - d(\mu))^*(\eta, n-d(\eta))(\zeta, n -
d(\zeta))^* \in B_n.\] This follows from the argument of \cite[Lemma~5.4]{kp2000}. To
wit, we have $(\mu, n - d(\mu))^*(\eta, n-d(\eta)) = 0$ unless $r(\mu, n - d(\mu)) =
r(\eta, n-d(\eta))$, which in turn forces $d(\mu) = d(\eta)$. But then $\overline{d}(\mu,
n - d(\mu)) = \overline{d}(\eta, n-d(\eta))$, and then the Cuntz--Krieger relation forces
$(\mu, n - d(\mu))^*(\eta, n-d(\eta)) = \delta_{\mu,\eta} (s(\mu), n)$. Hence \[(\lambda,
n-d(\lambda))(\mu, n - d(\mu))^*(\eta, n-d(\eta))(\zeta, n - d(\zeta))^* =
\delta_{\mu,\eta} (\lambda, n-d(\lambda))(\zeta, n - d(\zeta))^* \in B_n.\] For each
$v\in \Lambda^0$, $M_{\overline{\Lambda} (v, n)}(K)\cong M_{\Lambda v}(K)$. So the
elements $(\lambda, n-d(\lambda))(\mu, n - d(\mu))^*$ satisfy the same multiplication
formula as the matrix units $\mathbf{e}_{\lambda,\mu}$ in $\bigoplus_{v \in \Lambda^0}
M_{\Lambda v}(K)$. Hence the uniqueness of the latter shows that there is an isomorphism
as claimed.
\end{proof}

\begin{lem}\label{lem:dirlim}
For $m \le n \in \mathbb{Z}^k$, we have $B_m \subseteq B_n$, and in particular for each
$v \in \Lambda^0$, we have $(v,m) = \sum_{\alpha \in v\Lambda^{n-m}} (\alpha,
m)(\alpha,m)^*$.
\end{lem}
\begin{proof}
Again, this follows from the proof of \cite[Lemma~5.4]{kp2000}. We just apply the
Cuntz--Krieger relation, using at the first equality that $\Lambda$ has no sources:
\begin{align*}
(\lambda, m-d(\lambda))(\mu, m - d(\mu))^*
    &= (\lambda, m-d(\lambda))\Big(\sum_{\alpha \in s(\lambda)\Lambda^{n-m}} (\alpha, m)(\alpha,m)^*\Big)(\mu, m - d(\mu))^*\\
    &= \sum_{\alpha \in s(\lambda)\Lambda^{n-m}} (\lambda\alpha, m-d(\lambda))(\mu\alpha, m-d(\mu))^* \in B_n.
\end{align*}
This gives the first assertion, and the second follows by taking $\lambda = \mu = v$.
\end{proof}

\begin{thm}\label{poryt}
Let $\Lambda$ be a row-finite $k$-graph with no sources and $K$ a field. Then the
Kumjian--Pask algebra $\KP_K(\Lambda)$ is a graded von Neumann regular ring.
\end{thm}
\begin{proof}
Lemma~\ref{grneumann} shows that $\KP_K(\Lambda)$ is graded regular if and only if
$\KP_K(\Lambda)\#\mathbb{Z}^k$ is graded regular. By~(\ref{ttteeett})
$\KP_K(\Lambda)\#\mathbb{Z}^k \cong \KP_K(\overline{\Lambda})$ and the latter is an
ultramatricial algebra by Lemma~\ref{lem:dirlim}. Since ultramatricial algebras are
regular, the theorem follows.
\end{proof}

Since $\KP_K(\Lambda)$ is graded von Neumann regular, we immediately obtain the following
statements.

\begin{thm}\label{poryt1}
Let $\Lambda$ be a row-finite $k$-graph with no sources and $K$ a field. Then the
Kumjian--Pask algebra $A=\KP_K(\Lambda)$ has the following properties:
\begin{enumerate}[\upshape(1)]
\item any finitely generated right (left) graded ideal of $A$ is generated by one
    homogeneous idempotent;
\item any graded right (left) ideal of $A$ is idempotent;
\item any graded ideal is graded semi-prime;
\item $J(A)=J^{\gr}(A) =0$; and
\item there is a one-to-one correspondence between the graded right (left) ideals of
    $A$ and the right (left) ideals of $A_0$.
\end{enumerate}
\end{thm}
\begin{proof}
All the assertions are the properties of a graded von Neumann regular ring
\cite[\S1.1.9]{haz}, so the result follows from Theorem~\ref{poryt}.
\end{proof}

For the next result, given a $k$-graph $\Lambda$, and given $m \le n \in \mathbb{Z}^k$,
we define $\phi_{m,n} : \mathbb{N}\Lambda^0 \to \mathbb{N}\Lambda^0$ by $\phi_{m,n}(v) =
\sum_{w \in \Lambda^0} |v\Lambda^{n-m} w| w$.

\begin{cor}\label{cor:semigroup}Let $\Lambda$ be a row-finite $k$-graph with no sources and $K$ a field. There is an isomorphism
\[\textstyle
\VV(\KP_{K}(\overline{\Lambda})) \cong \varinjlim_{\mathbb{Z}^k} \big(\mathbb{N}\Lambda^0, \phi_{m,n})
\]
that carries $[(v,n)]$ to the copy of $v$ in the $n$th copy of $\mathbb{N}\Lambda^0$.
Fathermore, the monoid $\VV(\KP_{K}(\overline{\Lambda}))$ is cancellative.
\end{cor}
\begin{proof}
It is standard that there is an isomorphism $\VV\big(\bigoplus_{v \in \Lambda^0}
M_{\Lambda v}(K)\big) \cong \mathbb{N}\Lambda^0$ that takes
$\mathbf{e}_{\lambda,\lambda}$ to $s(\lambda)$ for all $\lambda$. So
Lemma~\ref{lem:summands} implies that there is an isomorphism $\VV(B_n) \to
\mathbb{N}\Lambda^0$ that carries $[(\lambda, n-d(\lambda))(\lambda, n-d(\lambda))^*]$ to
$s(\lambda)$ for all $\lambda$. Let $S_n$ be a copy $\mathbb{N}\Lambda^0\times \{n\}$ of
the monoid $\mathbb{N}\Lambda^0$ (so $(a,n) + (b,n) = (a+b,n)$ in $S_n$). Lemma~\ref{lem:dirlim} shows that these isomorphisms of monoids carry the inclusions $B_m
\hookrightarrow B_n$ to the maps $(v, m) \mapsto \sum_{\lambda \in v\Lambda^{n-m}}
(s(\lambda), n)$, which is precisely given by the formula $\phi_{m,n}$ for $m\leq
n\in\mathbb{Z}^k$. Since the monoid of a direct limit is the direct limit of the monoids
of the approximating algebras, we have an isomorphism $\VV(\KP_{K}(\overline{\Lambda}))
\cong \varinjlim_{\mathbb{Z}^k} S_n$, which sends $[(v, n)]$ to $(v,n)\in S_n$.

Suppose that $x+z=y+z$ in $\VV(KP_K(\overline{\Lambda}))$. By the isomorphism
$\VV(\KP_{K}(\overline{\Lambda})) \cong \varinjlim_{\mathbb{Z}^k} S_n$, there exist
images $x',y', z'$ of $x, y, z$, respectively, in $S_{n_0}=\mathbb{N}\Lambda^0\times
\{n_0\}$ for some $n_0\in\mathbb{Z}^k$ such that $x'+z'=y'+z'$. The monoid
$\mathbb{N}\Lambda^0$ is cancellative, so $\VV(KP_K(\overline{\Lambda}))$ is too.
\end{proof}

\begin{cor} Let $\Lambda$ be a row-finite $k$-graph with no sources and $K$ a field. Then $\VV^{\gr}(\KP_K(\Lambda)) \cong \varinjlim_{\mathbb{Z}^k}
\big(\mathbb{N}\Lambda^0, \phi_{m,n})$.
\end{cor}
\begin{proof} Recall from \eqref{ttteeett} that $\KP_K(\overline{\Lambda}) \cong
\KP_K(\Lambda)\#\mathbb{Z}^k$. Specialising Proposition~\ref{kikiki} to Kumjian--Pask
algebras, we have the isomorphism of categories $\Psi:
KP_K(\Lambda)\-\Gr\xrightarrow[]{\sim}KP_K(\overline{\Lambda})\-\Mod$. We argue as in the
directed-graph situation that $\Psi$ preserves finitely generated projective objects. By
\eqref{monoididempotent} and \eqref{grlocal}, we have $\VV^{\gr}(\KP_K(\Lambda)) \cong
\VV(\KP_K(\overline{\Lambda}))$.
\end{proof}

\section{The graded representations of the Steinberg algebra}\label{sectionseven}
In this section,  for a $\Gamma$-graded groupoid $\mg$ and its associated Steinberg
algebra $A_R(\mg)$, we construct graded simple $A_R(\mg)$-modules. Specialising our
results to the trivial grading, we obtain irreducible representations of (ungraded)
Steinberg algebras. We determine the ideals arising from these representations and prove
that these ideals relate to the effectiveness or otherwise of the groupoid.

\subsection{Representations of a Steinberg algebra}
Let $\mg$ be an ample Hausdorff groupoid, let $\G$ be a discrete group with identity
$\varepsilon$, and let $c : \mg \to \G$ be a continuous $1$-cocycle. A subset $U$ of the
unit space $\mg^{(0)}$ of $\mg$ is \emph{invariant} if $d(\g)\in U$ implies $r(\g)\in U$;
equivalently,
\[
    r(d^{-1}(U))=U=d(r^{-1}(U)).
\]

Given an element $u\in \mg^{(0)}$, we denote by $[u]$ the smallest invariant subset of
$\mg^{(0)}$ which contains $u$. Then $$r(d^{-1}(u))=[u]=d(r^{-1}(u)).$$ That is, for any
$v\in [u]$, there exists $x\in\mg$ such that $d(x)=u$ and $r(x)=v$; equivalently, for any
$w\in [u]$, there exists $y\in\mg$ such that $d(y)=w$ and $r(y)=u$. Thus for any
$v,w\in[u]$, there exists $x\in\mg$ such that $d(x)=v$ and $r(x)=w$. We call $[u]$ an
\emph{orbit}. Observe that an invariant subset $U\subseteq \mg^{(0)}$ is an orbit if and
only if for any $v,w\in U$, there exists $x\in\mg$ such that $d(x)=v$ and $r(x)=w$.

\begin{lem}\label{disjoint}
Let $u_{1}, u_{2}, \cdots, u_{n}$ be pairwise distinct elements of $\mg^{(0)}$ with
$n\geq 2$. Then there exist disjoint compact open bisections $B_{i}\subseteq \mg^{(0)}$
such that $u_{i}\in B_{i}$ for each $i=1, \cdots, n$.
\end{lem}
\begin{proof}
Since $\mg^{(0)}$ is a Hausdorff space, there exist disjoint open subsets $X_{i}$ of
$\mg^{(0)}$ such that $u_{i}\in X_{i}$ for all $i$. Since $\mg$ is ample, we can choose
compact open bisections $B_{i}\subseteq X_{i}$ such that $u_{i}\in B_{i}$ for all $i$.
\end{proof}

The \emph{isotropy group} at a unit $u$ of $\mg$ is the group  $\Iso(u) =\{\g\in\mg \mid
d(\g)=r(\g)=u\}.$ A unit $u\in\mg^{(0)}$ is called \emph{$\Gamma$-aperiodic} if $\Iso(u)
\subseteq c^{-1}(\varepsilon)$, otherwise $u$ is called \emph{$\Gamma$-periodic}. For an
invariant subset $W\subseteq \mg^{(0)}$, we denote by $W_{\rm ap}$ the collection of
$\Gamma$-aperiodic elements of $W$ and by $W_{\rm p}$ the collection of $\Gamma$-periodic
elements of $W$. Then
$$W=W_{\rm ap} \bigsqcup W_{\rm p}.$$
If $W=W_{\rm ap}$, we say that $W$ is \emph{$\Gamma$-aperiodic}; If $W=W_{\rm p}$, we say
that $W$ is \emph{$\Gamma$-periodic}.

\begin{rmk}
Let $E$ be a directed graph. Let $\mg_E$ be the associated graph groupoid and $c : \mg_E
\to \mathbb{Z}$ the canonical cocycle $c(x,m,y) = m$. It was shown in \cite{kprr} that
$c^{-1}(0)$ is a principal groupoid, in the sense that $\Iso(c^{-1}(0)) = \mg_E^{(0)}$.
Hence $x \in \mg_E^{(0)} = E^\infty$ is $\mathbb{Z}$-aperiodic if and only if $\Iso(x) =
\{x\}$. It is standard that $\Iso(x) = \{x\}$ if and only if $x \not= \mu \lambda^\infty$
for any cycle $\lambda$ in $E$. So $x$ is $\mathbb{Z}$-aperiodic if and only if $x \not= \mu
\lambda^\infty$ for any cycle $\lambda$.
\end{rmk}

\begin{lem}\label{lemmainv}
Let $W\subseteq \mg^{(0)}$ be an invariant subset. Then $W_{\rm ap}$ and $W_{\rm p}$ are
both invariant subsets of $\mg^{(0)}$.
\end{lem}
\begin{proof}
For $x\in \mg$, let $u=d(x)$ and $v=r(x)$.   Suppose that $u\in W_{\rm ap}$. If  $c(y)
\not= \varepsilon$ for some $y\in \Iso(v)$, then $x^{-1}yx\in \Iso(u)$ and $\varepsilon
\not= c(y) = c(x)c(x^{-1}yx)c(x)^{-1}$, forcing $c(x^{-1}yx) \not= \varepsilon$, a
contradiction. Hence, $v=r(x)$ is $\Gamma$-aperiodic. Since $W$ is invariant,  we have
$v\in W_{\rm ap}$. So $W_{\rm ap}$ is invariant. Since $W = W_{\rm ap} \sqcup W_{\rm p}$,
it follows that $W_{\rm p}$ is also invariant.
\end{proof}

By the proof of Lemma \ref{lemmainv},  $u\in \mg^{(0)}$ is $\Gamma$-aperiodic if and only
if its orbit $[u]$ is $\Gamma$-aperiodic.

\begin{exm} \label{exmap}
In this example we construct a $\mathbb{Z}$-aperiodic invariant subset which is neither
open nor closed in $\mg^{(0)}$. Let $E$ be the following directed graph.
$$
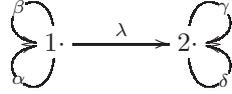
\xymatrix@C=3pc{
1\cdot \ar@(u, l)[]|{\b}
\ar@(d, l)[]|{\a}\ar[r]^{\lambda}& 2\cdot \ar@(u, r)[]|{\g}
\ar@(d, r)[]|{\delta}
}
$$
Let $u$ be the infinite path $\a\b\a^{2}\b\a^{3}\b\cdots$. Then $u$ is an element in
$\mg_{E}^{(0)}$. The orbit $[u]$ consists of all infinite paths tail equivalent to $u$.
So $\a^n u \in [u]$ for all $n \in \mathbb{N}$. The sequence $\a^n u$ converges to
$\a^\infty$, which does not belong to $[u]$. So $[u]$ is not closed. Similarly, the
points $u_n := \a\b\a^2\b\cdots\a^n\b\a^\infty$ all belong to $\mg^{(0)} \setminus [u]$,
but $u_n \to u$, so $[u]$ is not open. In particular, neither $[u]$ nor its complement is
the invariant subset of $\mg^{(0)}$ corresponding to any saturated hereditary subset of
$E^0$.
\end{exm}

We will employ $\Gamma$-aperiodic invariant subsets of $\mg^{(0)}$ to obtain graded
representations for the Steinberg algebra $A_{R}(\mg)$.  For any invariant subset
$U\subseteq \mg^{(0)}$ and a unital commutative ring $R$, we denote by $RU$ the free
$R$-module with basis $U$. For every compact open bisection $B\subseteq \mg$, there is a
function $f_{B}: \mg^{(0)}\xra RU$ which has support contained in $d(B)\cap U$ and
$f_{B}(d(\g))=r(\g)$ for all $\g\in B\cap d^{-1}(U)$. There is a unique representation
$\pi_{U}:A_{R}(\mg)\xra \End_R(RU)$ such that
\begin{equation}\label{moduleaction}
\pi_{U}(1_{B})(u)=f_{B}(u)
\end{equation}
for every compact open bisection $B$ and $u\in U$. This representation makes $RU$ an
$A_{R}(\mg)$-module (see \cite[Proposition 4.3]{bcfs}). An $A_{R}(\mg)$-submodule
$V\subseteq RU$ is called a \emph{basic submodule} of $RU$ if whenever $r \in
R\setminus\{0\}$ and $ru\in V$, we have $u\in V$. We say an $A_{R}(\mg)$-module is
\emph{basic simple} if it has no non-trivial basic submodules.

We can state one of the main results of this section.
\begin{thm} \label{basicsimple}
Let $U$ be an invariant subset of $\mg^{(0)}$. Then $U$ is a $\Gamma$-aperiodic orbit if
and only if $RU$ is a graded basic simple $A_{R}(\mg)$-module. Furthermore, $RU$ is a
graded basic simple $A_{R}(\mg)$-module if and only if it is graded and basic simple.
\end{thm}
\begin{proof}
Suppose that $u \in \mg^{(0)}$ satisfies $U=[u]$, and that $[u]$ is a $\Gamma$-aperiodic
orbit. We first show that $R[u]$ is a $\G$-graded $A_{R}(\mg)$-module. For any $\g\in\G$,
set
\[
[u]_{\g}=\{v\in [u] \mid \text{ there exists } x\in\mg \text{ such that } c(x) = \g, d(x)=u\text{ and }r(x)=v\}.
\]
We claim that $[u]_{\g}\cap [u]_{\g'} \not= \emptyset$ implies $\g = \g'$. Indeed, if
$v\in [u]_{\g}\cap [u]_{\g'}$, then there exist $x\in c^{-1}(\g)$ and $y\in c^{-1}(\g')$
such that $d(x)=d(y)=u$ and $r(x)=r(y)=v$. Now $x^{-1}y\in \Iso(u)$. Since $u$ is
$\Gamma$-aperiodic this forces $\gamma^{-1}\gamma' = c(x^{-1}y) = \varepsilon$, and so
$\g = \g'$. This gives a partition $[u]=\sqcup_{\g\in\G} [u]_{\g}$. Therefore
$A_{R}(\mg)$-module $R[u]$ has a decomposition of $R$-modules
\begin{equation*}
R[u]=\bigoplus_{\g\in\G}(R[u])_{\g},
\end{equation*} where $(R[u])_{\g}$ is a free $R$-module with basis $[u]_{\g}$.

We show that $A_{R}(\mg)_{\a}\cdot (R[u])_{\g}\subseteq (R[u])_{\a\g}$, for $\a,\g\in
\G$. Fix $v \in [u]_\gamma$ and $B \in B^{\rm co}_\a(\mg)$. We  use $\cdot$ to denote the
action of $A_{R}(\mg)$ on $RU$. We have
\begin{equation*}
1_{B}\cdot v=\begin{cases} r(b), & \text{ if $b \in B$ satisfies $d(b) = v$};\\
0,&\text{ if $v \not\in d(B)$.}
\end{cases}
\end{equation*}
Clearly $0 \in (R[u])_{\a\g}$, so suppose that $b \in B$ satisfies $d(b) = v$. Since
$v\in [u]_{\g}$, there exists $x\in\mg$ such that $c(x) = \g$, $d(x)=u$, and $r(x)=v$.
Now $d(bx)=u$, $r(bx)=r(b)$, and $c(bx) = c(b)c(x) = \a\g$. So $r(b)\in [u]_{\a\g}$.
Since elements of the form $1_B$ where $B \in B^{\rm co}_\a(\mg)$ span $A_{R}(\mg)_{\a}$,
we deduce that $A_{R}(\mg)_{\a}\cdot (R[u])_{\g}\subseteq (R[u])_{\a\g}$ as claimed.

Next we show that $R[u]$ is a basic simple $A_R(\mg)$-module. Suppose that $V\neq 0$ is a
basic $A_{R}(\mg)$-submodule of $R[u]$. Take a nonzero element $x\in V$. Fix nonzero
elements $r_i \in R$ and pairwise distinct $u_i \in [u]$ such that
$x=\sum_{i=1}^{m}r_{i}u_{i}$. By Lemma~\ref{disjoint}, there exist disjoint compact open
bisections $B_{i}\subseteq \mg^{(0)}$ such that $u_{i}\in B_{i}$ for all $i=1, \cdots,
m$. Now
\begin{equation*}
1_{B_{1}}\cdot x
    = 1_{B_{1}}\cdot \sum_{i=1}^{m}r_{i}u_{i}
    = \sum_{i=1}^{m}r_{i}(1_{B_{1}}\cdot u_{i})
    = r_{1}f_{B_{1}}(u_{1}).
\end{equation*}
Thus $u_1=f_{B_{1}}(u_{1})\in V$, because $V$ is a basic submodule. Fix $v \in [u]$ and
choose $x\in\mg$ such that $d(x)=u_1$ and $r(x)=v$. Fix a compact open bisection $D$
containing $x$. Then $1_{D}\cdot u_1=f_{D}(u_1)=r(x)=v\in V$, giving $V=R[u]$. Thus
$R[u]$ is basic simple, and consequently graded basic simple.

For the converse suppose that $RU$ is a graded basic simple $A_R(\mg)$-module. We first
show that $U$ is $\Gamma$-aperiodic. Let $u\in U$. We claim that there exists $r \in
R\setminus\{0\}$ such that $ru$ is a homogeneous element of $RU$. To see this, express
$u=\sum_{i=1}^{l}h_{i}$, where $h_{i}\neq u$ are homogeneous elements. For each $i$,
express $h_{i}=\sum_{j=1}^{s_{i}}\lambda_{ij}u_{ij}$ with $\lambda_{ij} \in
R\setminus\{0\}$ and the $u_{ij}\in U$ pairwise distinct. We first show that $u \in
\{u_{ij} \mid  i=1, \cdots, l; \; j=1,\cdots, s_{i}\}$; for if not, then
Lemma~\ref{disjoint} gives compact open bisections $B, B_{ij}$ such that $u\in B$ and
$u\notin B_{ij}$ for all $i, j$. So $1_B \cdot u\neq 0$, whereas
\[
1_{B}\cdot u
    = 1_{B}\cdot \sum_{i=1}^{l}h_{i}
    = 1_{B}\cdot\sum_{i=1}^{l}\sum_{j=1}^{s_{i}}\lambda_{ij}u_{ij}
    = \sum_{i=1}^{l}\sum_{j=1}^{s_{i}}\lambda_{ij}1_{B}\cdot u_{ij}=0.
\]
This is a contradiction. So $u = u_{ij}$ for some $i,j$ as claimed; without loss of
generality, $u = u_{11}$. Hence $h_{1} = \lambda_{11}u +
\sum_{j=2}^{s_{1}}\lambda_{1j}u_{1j}$. There exist compact open bisections $B', B_{1j}'
\subseteq \mg^{(0)}\subseteq c^{-1}(\varepsilon)$ such that $u\in B'$ but $u\notin
B_{1j}'$ for $j\neq 1$. Hence $r := \lambda_{11}$ belongs to $R \setminus \{0\}$, and
$$r u = \lambda_{11}1_{B'} \cdot u = 1_{B'}\cdot h_{1}$$ is homogeneous as claimed. Now
suppose that $u$ is not $\Gamma$-aperiodic. Then there exists $x\in \Iso(u)$ with $c(x)
\not= \varepsilon$. Fix $D \in B^{\rm co}_{c(x)}(\mg)$ containing $x$. Then $1_{D}\cdot
ru = r 1_D\cdot u = ru$ is homogeneous. Thus $1_D \in A_R(\mg)_\varepsilon$, forcing
$c(x) = \varepsilon$. This is a contradiction. Thus $U$ is $\Gamma$-aperiodic.

For the last part of the theorem we prove that $U$ is an orbit. If not then there exist
$u,v \in \mg^{(0)}$ with $[u] \cap [v] = \emptyset$ and $[u] \sqcup [v] \subseteq U$. Hence
$R[u] \subseteq RU \setminus R[v]$ is a nontrivial proper graded basic submodule of $RU$
by the first part of the theorem. This is a contradiction. So $U$ is an orbit. The last
statement of the theorem follows from the first part of the proof.
\end{proof}

\begin{cor} \label{corforsimple}
Let $\mg$ be an ample Hausdorff groupoid. $U$ be an invariant subset of $\mg^{(0)}$. Then
$U$ is an orbit of $\mg^{(0)}$ if and only if $RU$ is a basic simple $A_{R}(\mg)$-module.
\end{cor}
\begin{proof}
Apply Theorem~\ref{basicsimple} with $c : \mg \to \{\varepsilon\}$ the trivial grading.
\end{proof}

Specialising Theorem~\ref{basicsimple} to the case of Leavitt path algebras we obtain
irreducible representations for these algebras.

Let $K$ be a field. For an infinite path $p$ in a graph $E$, Chen constructed the left
$L_K(E)$-module $\mathcal{F}_{[p]}$ of the space of infinite paths tail-equivalent to $p$
and proved that it is an irreducible representation of the Leavitt path algebra (see
\cite[Theorem 3.3]{c}). These were subsequently called Chen simple modules and further
studied in~\cite{abramsman,araranga,araranga2,hr,ranga}. In the groupoid setting, the
infinite path $p$ is an element in $\mg_{E}^{(0)}$. Thus $q$ belongs to the orbit $[p]$
if and only if $q$ is tail-equivalent to $p$. Applying Corollary \ref{corforsimple}, we
immediately obtain that $K[p]=\mathcal{F}_{[p]}$ is an irreducible representation of the
Leavitt path algebra. Furthermore, by Theorem~\ref{basicsimple}, $p$ is an aperiodic
infinite path (irrational path) if and only if $\mathcal{F}_{[p]}$ is a graded module
(see \cite[Proposition 3.6]{hr}).

Recall from \cite[Theorem 3.3]{c} that $\End_{L_{K}(E)}(\mathcal{F}_{[p]})\cong K$. We
claim that $\End_{A_{R}(\mg)}(R[u])\cong R$ for $u\in \mg_{E}^{(0)}$. Indeed, let $f :
R[u] \xra R[u]$ be a nonzero homomorphism of $A_{R}(\mg)$-modules. Then ${\rm Ker} f$ is
a basic submodule of $R[u]$. Since $R[u]$ is basic simple, we deduce that $f$ is
injective. For $v\in [u]$, we write $f(v)=\sum_{i=1}^{n}r_{i}v_{i}$ with $0\neq r_{i}\in
R$ and $v_{i}$ are distinct. We prove that $n=1$ and $v=v_{1}$. For if not, then we may
assume that $v\neq v_{1}$. By Lemma \ref{disjoint}, there exist disjoint compact open
bisections $B, B_{1}\subseteq \mg^{(0)}$ such that $v\in B$, $v_{1}\in B_{1}$ and
$v_{i}\notin B_{1}$ for $i\neq 1$. Then $1_{B_{1}}\cdot f(v)=f(1_{B_{1}}\cdot v)=0$. But,
$1_{B_{1}}\cdot f(v)=1_{B_{1}}\cdot \sum_{i=1}^{n}r_{i}v_{i}=r_{1}v_{1}$ which is a
contradiction.

Likewise, Theorem~\ref{basicsimple} specialises to $k$-graph groupoids, giving new
information about Kumjian--Pask algebras.

\begin{cor}
Let $\Lambda$ be a row-finite $k$-graph without sources and $\KP_K(\Lambda)$ the
Kumjian--Pask algebra of $\Lambda$. Then
\begin{enumerate}[\upshape(1)]
\item for an infinite path $x\in \Lambda^\infty$, $K[x]$ is a simple left
    $\KP_K(\Lambda)$-module;
\item for $x,y \in \Lambda^\infty$, we have $K[x]\cong K[y]$  if and only if $x\sim
    y$; and
\item for $x \in \Lambda^\infty$, $K[x]$ is a graded module if and only if $x$ is an
    aperiodic path.
\end{enumerate}
\end{cor}
\begin{proof}
For (1), the equivalence class of $x$ is the orbit of $\mg_{\Lambda}^{(0)}$ which
contains $x$. By \eqref{moduleaction} and Corollary \ref{corforsimple}, the statement
follows directly.  For (2), let $\phi: F([x])\rightarrow F([y])$ be an isomorphism. Write
$\phi(x)=\sum_{i=1}^l r_i y_i$, where $y_i\sim y$ are all distinct. If $x=y_i$, for some
$i$, then by transitivity of $\sim$, $x\sim y$ and we are done. Otherwise one can choose
$n\in \mathbb N^k$ such that all $y_i(0,n)$ and $x(0,n)$ are distinct. Setting
$a=y_1(0,n)$, we have $0=\phi(a^*x)=a^*\phi(x)=y_1(n,\infty)$, which is not possible
unless $x=y_1$ and $l=1$. This gives that $x\sim y$. The converse is clear. The statement
(3) follows immediately by Theorem \ref{basicsimple}.
\end{proof}

\subsection{The annihilator ideals and effectiveness of groupoids} \label{subsection72}
In this section, we describe the annihilator ideals of the graded modules over a
Steinberg algebra and prove that these ideals reflect the effectiveness of the groupoid.

As in previous sections, we assume that $\mg$ is a $\G$-graded ample Hausdorff groupoid
which has a basis of graded compact open bisections. Let $R$ be a commutative ring with
identity and $A_R(\mg)$ the $\G$-graded Steinberg algebra associated to $\mg$.

Let $W\subseteq \mg^{(0)}$ be an invariant subset.  We write $\mg_{W} := d^{-1}(W)$ which
coincides with the restriction $\mg|_{W}=\{x\in\mg \mid  d(x)\in W,  r(x)\in W\}$. Notice
that  $\mg_{W}$ is a groupoid with unit space $W$.

Observe that the interior $W^\circ$ of an invariant subset $W$ is invariant. Indeed,
$r(d^{-1}(W^\circ))$ is an open subset of $\mg^{(0)}$, since $W^\circ$ is an open subset
of $\mg^{(0)}$. Since $W$ is invariant, $r(d^{-1}(W^\circ))\subseteq W$. Thus
$r(d^{-1}(W^\circ))\subseteq W^\circ$. It follows that the closure $W^{-}$ of $W$ is also
an invariant subset of $\mg^{(0)}$, since $W^{-}=\mg^{(0)}\setminus (\mg^{(0)}\setminus
W)^\circ$.

Recall from \eqref{moduleaction} that $$\pi_{W}:A_{R}(\mg)\longrightarrow \End_R(RW)$$
makes $RW$ an $A_{R}(\mg)$-module.

\begin{lem}\label{lemmaann}
Let $W\subseteq \mg^{(0)}$ be an invariant subset of the unit space of $\mg$, and let $U
= (\mg^{(0)}\setminus W)^\circ$. Then
$$A_{R}(\mg_{U})\subseteq \Ann_{A_{R}(\mg)}(RW).$$
\end{lem}
\begin{proof}
For any $f\in A_{R}(\mg_{U})$, we write $f=\sum_{k=1}^{m}r_{k}1_{B_{k}}$  with
$B_{k}\subseteq \mg_{U}$ compact open bisections of $\mg$ and $r_{k}\in R$ nonzero
scalars. Since $d(B_{k})\subseteq U$, we have $d(B_{k})\cap W=\emptyset$. Thus $f\cdot
w=0$ for any $w\in W$, and hence $f\in \Ann_{A_{R}(\mg)}(RW)$.
\end{proof}

From now on, $W\subseteq \mg^{(0)}$ is a $\Gamma$-aperiodic invariant  subset. We have
$$W=\bigcup_{u\in W}[u].$$
Of course, two elements of $W$ may belong to the same orbit.

Recall from Theorem~\ref{basicsimple} that if $u\in\mg^{(0)}$ is $\Gamma$-aperiodic, then
$R[u]$ is a $\G$-graded $A_{R}(\mg)$-module.  Therefore $RW$ is a $\G$-graded
$A_{R}(\mg)$-module. In order to construct graded representations for $A_R(\mg)$, we need
to consider the ``closed'' subgroups of $\End_{R}(FW)$ defined in~(\ref{moduleaction}).
Namely, we consider the subgroup $\END_R(RW)=\bigoplus_{\g\in \G}{\rm Hom}_{R}(RW,
RW)_{\g}$, where each component $\Hom_{R}(RW, RW)_{\g}$ consists of $R$-maps of degree
$\g$.

Then the map
\begin{equation}\label{gradedhom}
\pi_{W}: A_{R}(\mg)\longrightarrow \END_R(RW)
\end{equation}
given by the $A_{R}(\mg)$-module action is a homomorphism of $\G$-graded algebras. To
prove that $\pi_{W}$ preserves the grading, fix $\a\in \G$ and $B \in B^{\rm
co}_\a(\mg)$. Take $u \in W$ and $v\in [u]$. Fix $x \in \mg$ with $d(x) = u$ and $r(x) =
v$, and put $\b = c(x)$ so that $v \in [u]_\b$. Then
\begin{equation*}
\pi_{W}(1_{B})(v)=
    \begin{cases}
        r(\g) & \text{if } v=d(\g) \text{ for some } \g\in B;\\
        0 & \text{otherwise}.
    \end{cases}
\end{equation*}
Since $c(\g x) = \a\b$, we obtain $\pi_{W}(1_{B})\in \Hom_{R}(RW, RW)_{\a}$.

Recall that an ample Hausdorff groupoid $\mg$ is \emph{effective} if
$\Iso(\mg)^\circ=\mg^{(0)}$, where $\Iso(\mg)=\bigsqcup_{u\in\mg^{(0)}}\Iso(u)$. It
follows that $\mg$ is effective if  and only if for any nonempty $B \in B^{\rm
co}_*(\mg)$ with $B \cap \mg^{(0)} = \emptyset$, we have $B \not \subseteq \Iso(\mg)$
(see~\cite[Lemma~3.1]{bcfs} for other equivalent conditions).

We need the following graded uniqueness theorem for Steinberg algebras established
in~\cite[Theorem 3.4]{cm}.

\begin{lem} \label{gut}
Let $\mg$ be a $\Gamma$-graded ample Hausdorff groupoid such that $c^{-1}(\varepsilon)$
is effective. If $\pi:A_{R}(\mg)\xra A$ is a graded $R$-algebra homomorphism with
$\Ker(\pi)\neq 0$ then there is a compact open subset $B\subseteq \mg^{(0)}$ and $r\in
R\setminus \{0\}$ such that $\pi(r1_{B})=0$.
\end{lem}

The following key lemma will be used to determine the annihilator ideal of the
$A_{R}(\mg)$-module $RW$. This is a generalisation of~\cite[Proposition 4.4]{bcfs}
adapted to the graded setting. Recall that if $\mg$ is a graded groupoid with grading
given by the continuous $1$-cocycle $c : \mg \to \G$, then $c^{-1}(\varepsilon)$ is a
(trivially graded) clopen subgroupoid of $\mg$.

\begin{lem} \label{keylemma}
Let $W\subseteq \mg^{(0)}$ be a $\Gamma$-aperiodic invariant  subset and $\pi_{W}:
A_{R}(\mg)\xra\END_R(RW)$ the homomorphism of $\G$-graded algebras given in
\eqref{gradedhom}. Then $\pi_{W}$ is injective if and only if $W$ is dense in $\mg^{(0)}$
and $c^{-1}(\varepsilon)$ is effective.
\end{lem}
\begin{proof} Suppose $\pi_{W}$ is injective and  there exists an open subset $K$ of
$\mg^{(0)}$ such that $K\cap W=\emptyset$.  We have $K=\bigcup_{i} B_{i}$, where  $B_{i}$
are compact open bisections of $\mg$. So $B_{i}\cap W=\emptyset$ for each $i$, giving
$\pi_{W}(1_{B_{i}})=0$, a contradiction. Thus for any open subset $K$ of $\mg^{(0)}$,
$K\cap W\neq \emptyset$. Therefore $W$ is dense in $\mg^{(0)}$.

Suppose now that $c^{-1}(\varepsilon)$ is not effective. Then there exists a nonempty
compact open bisection $B\subseteq c^{-1}(\varepsilon) \setminus \mg^{(0)}$ such that
$d(b)= r(b)$ for all $b \in B$. We have that $d(B)\neq B$ and that $B$ is a compact open
bisection of $\mg$. Thus $1_{B}-1_{d(B)}\in \Ker (\pi_{W})$. This is a contradiction.
Hence, $c^{-1}(\varepsilon)$ is effective.

For the converse, Lemma~\ref{gut} implies that it suffices to prove that for any compact
open subset $B\subseteq \mg^{(0)}$ and $r\in R\setminus \{0\}$, $\pi_{W}(r1_{B})\neq 0$.
Since $W$ is dense in $\mg^{(0)}$,  we have $B\cap W\neq \emptyset$. There exists $w\in
B\cap W$ such that $\pi_{W}(r1_{B})(w)\neq 0$, proving $\pi_{W}(r1_{B})\neq 0$.
\end{proof}

If the group $\G$ is trivial, then by Lemma \ref{keylemma}, for an invariant subset
$W\subseteq \mg^{(0)}$, the homomorphism $\pi_{W}:A_{R}(\mg)\xra \End_{R}(RW)$ is
injective if and only if $W$ is dense in $\mg^{(0)}$ and the groupoid $\mg$ is effective.

The following is the main result of this section.

\begin{thm}\label{eqthm} Let $\mg$ be a $\Gamma$-graded ample  Hausdorff groupoid, $R$ a commutative ring with identity and $A_R(\mg)$ the Steinberg algebra associated to $\mg$. The following statements are equivalent:
\begin{enumerate}
\item[(i)] Let $W\subseteq \mg^{(0)}$ be a $\Gamma$-aperiodic invariant  subset and
    $W^{-}$ the closure of $W$. Then the groupoid $\big(c|_{\mg_{W^-}}\big)^{-1}(\varepsilon)$ is effective;

\medskip

\item[(ii)] For any $\Gamma$-aperiodic invariant  subset $W\subseteq \mg^{(0)}$,
\[\Ann_{A_{R(\mg)}}(RW)=A_{R}(\mg_{U}),\] where $U=(\mg^{(0)}\setminus W)^\circ$ is the interior of the invariant subset $\mg^{(0)}\setminus W$.
\end{enumerate}
\end{thm}

\begin{proof} $(i)\Rightarrow (ii)$
Let $W\subseteq \mg^{(0)}$ be a $\Gamma$-aperiodic invariant  subset. By
Theorem~\ref{basicsimple}, $RW$ is a graded $A_R(\mg)$-module. By Lemma \ref{lemmaann},
we have $A_{R}(\mg_{U})\subseteq \Ann_{A_{R}(\mg)}(RW)$ with $U=(\mg^{(0)}\setminus
W)^\circ$. It follows that $RW$ is an $A_{R}(\mg)/A_{R}(\mg_{U})$-module.  By \cite[Lemma
3.6]{cmhs}, we have an exact sequence of canonical ring homomorphisms
\[
    0 \longrightarrow A_R(\mg_U) \longrightarrow A_R(\mg) \longrightarrow A_R(\mg_D)\longrightarrow 0.
\]
The homomorphisms are induced by extensions from $\mg_U$ to $\mg$ and restrictions from
$\mg$ to $\mg_D$, respectively. One can easily check that the homomorphisms are graded.
It therefore follows that the quotient algebra $A_{R}(\mg)/A_{R}(\mg_{U})$ is graded
isomorphic to $A_{R}(\mg_{D})$, where $D=\mg^{(0)}\setminus U$. It follows that $RW$ is a
$\G$-graded $A_{R}(\mg_{D})$-module (this also follows from Theorem~\ref{basicsimple}).
We denote by $\widehat{\pi}_{W}:A_{R}(\mg_{D})\xra{\END}_R(RW)$ the induced graded
homomorphism. Observe that $(\mg_D)^{(0)}=D$ is the closure of $W$. Thus by Lemma
\ref{keylemma}, the homomorphism $\widehat{\pi}_{W}$ is injective. This implies that $RW$
is a faithful $A_{R}(\mg_{D})$-module. Hence, the annihilator ideal of $RW$ as an
$A_{R}(\mg)$-module is $A_{R}(\mg_{U})$.

$(ii)\Leftarrow (i)$ Let $D$ denote the closure of $W$ in $\mg^{(0)}$. Then $RW$ is a
faithful $A_{R}(\mg_{D})$-module. So the result follows from Lemma~\ref{keylemma}.
\end{proof}

Recall that a groupoid $\mg$ is \emph{strongly effective} if for every nonempty closed
invariant subset $D$ of $\mg^{(0)}$, the groupoid $\mg_{D}$ is effective.

\begin{rmk} \begin{enumerate}
\item[(1)] If $c^{-1}(\varepsilon)$ is strongly effective, then Theorem \ref{eqthm}(i)
    holds. In fact, a closed invariant subset $D$ of the unit space of $\mg$ is in
    particular a closed $c^{-1}(\varepsilon)$-invariant subset of $\mg^{(0)}$. We have
    $c^{-1}(\varepsilon)_{D} = c^{-1}(\varepsilon) \cap \mg_{D} =
    \big(c|_{\mg_D}\big)^{-1}(\varepsilon)$. Hence, Theorem \ref{eqthm}(i) follows directly.
    Example~\ref{notseffective} below, on the other hand, shows that Theorem~\ref{eqthm}(i)
    does not imply that $c^{-1}(\varepsilon)$ is strongly  effective.

\item[(2)] Resume the notation of Example~\ref{exmap}, so $u =
    \alpha\beta\alpha^2\beta\cdots \in E^\infty$. Let $D$ be the closure of the
    $\mathbb{Z}$-aperiodic invariant subset $[u]\subseteq
    \mg_{E}^{(0)}$. As we saw in that example, $D$ is not itself $\mathbb{Z}$-aperiodic, because it contains $\a^\infty$.
\end{enumerate}
\end{rmk}

\begin{exm} \label{notseffective}
It is easy to construct examples of $\G$-graded groupoids with no $\G$-aperiodic points.
For example, let $X$ be the Cantor set. Regard $\mg=X\times \mathbb{Z}^{2}$ as a groupoid
with unit space $X \times \{0\}$ identified with $X$ by setting $r(x,m) = x = d(x,m)$ and
defining composition and inverses by $(x, n)(x, m)=(x, m+n)$ and $(x, m)^{-1} = (x, -m)$.
The map $c:\mg\xra \mathbb{Z}$ given by $c(x, (m_{1}, m_{2}))=m_{1}$ is a continuous
$1$-cocycle. We have $c^{-1}(0) = X\times (\{0\}\times \mathbb{Z})$, which is not
effective (for example $X \times \{(0,1)\}$ is a compact open bisection contained in the
isotropy subgroupoid of $c^{-1}(0)$). Moreover, $\mg^{(0)}$ has no $\mathbb{Z}$-aperiodic
points because $\{u\} \times (\mathbb{Z} \times \{0\}) \subseteq \Iso(u) \setminus
c^{-1}(0)$ for all $u \in \mg^{(0)}$; so every $u \in \mg^{(0)}$ is
$\mathbb{Z}$-periodic.
\end{exm}

Applying Theorem~\ref{eqthm} to the trivial grading, we obtain a new characterisation of
strong effectiveness.

\begin{cor}
Let $\mg$ be an ample Hausdorff groupoid, and $R$ be a commutative ring with identity.
Then $\mg$ is strongly effective if and only if for any invariant subset $W$ of
$\mg^{(0)}$, the annihilator of the $A_{R}(\mg)$-module $RW$ is $A_{R}(\mg_{U})$, where
$U=(\mg^{(0)}\setminus W)^\circ$.
\end{cor}

\section{Acknowledgements}
The authors would like to acknowledge Australian Research Council grants DP150101598 and
DP160101481. The first-named author was partially supported by DGI-MINECO (Spain) through
the grant MTM2014-53644-P.

\end{document}